\documentclass[a4paper,11pt]{amsart}
\usepackage{amsmath,inputenc,euscript,amssymb,geometry}
\usepackage{graphicx}
\usepackage{amssymb}
\usepackage{latexsym}
\usepackage{amssymb,amsbsy,amsmath,amsfonts,amssymb,amscd}
\usepackage{hyperref}

\newtheorem{lemma}{Lemma}[section]
\newtheorem{theorem}[lemma]{Theorem}
\newtheorem{prop}[lemma]{Proposition}
\newtheorem{corollary}[lemma]{Corollary}

\newtheorem{remark}[lemma]{Remark}

\newcommand{\ls}{\lesssim}

\newcommand{\D}{\displaystyle}

\begin{document}

\title[Scattering for 1D cubic NLS and singular vortex dynamics]{Scattering for 1D cubic NLS and \\ singular vortex dynamics}
\author[V. Banica]{Valeria Banica}
\address[V. Banica]{D\'epartement de Math\'ematiques\\ Universit\'e
  d'Evry\\ France} 
\email{Valeria.Banica@univ-evry.fr}
\author[L. Vega]{Luis Vega}
\address[L. Vega]{Departamento de Matematicas, Universidad del Pais Vasco, Spain} 
\email{luis.vega@ehu.es}


\def\Tend{{\longrightarrow}}
\def\tend{{\rightarrow}}
\newcommand{\R}{\mathbb{R}}
\newcommand{\nrL}[2]{{\| #1 \|_#2}}
\newcommand{\pq}[1]{{\| #1 \|_{L^p(\R,L^q(\R^n))}}}

\begin{abstract}In this paper we study the stability of the self-similar solutions of the binormal flow, which is a model for the dynamics of vortex filaments in fluids and super-fluids. These particular solutions $\chi_a(t,x)$ form a family of evolving regular curves of $\mathbb R^3$ that develop a singularity in finite time, indexed by a parameter $a>0$. 
We consider curves that are small regular perturbations of $\chi_a(t_0,x)$ for a fixed time $t_0$. In particular, their curvature is not vanishing at infinity, so we are not in the context of known results of local existence for the binormal flow. Nevertheless, we construct in this article solutions of the binormal flow with these initial data. 
 Moreover, these solutions become also singular in finite time. Our approach uses the Hasimoto transform what leads us to study the long-time behavior of a 1D cubic NLS equation with time-depending coefficients and small regular perturbations of the constant solution as initial data. We prove asymptotic completeness for this equation in appropriate function spaces. 
\end{abstract}

\maketitle


\tableofcontents

\section{Introduction}

In this work we complete the stability properties obtained in our previous paper \cite{BV} of the selfsimilar solutions of the binormal flow of curves
\begin{equation}\label{binormal}
\chi_t=\chi_x\land\chi_{xx}.
\end{equation}
Here $\chi=\chi(t,x)\in\mathbb R^3$, $x$ denotes the arclength parameter and $t$ the time variable. Using the Frenet frame, the above equation can be written as
\begin{equation*}\label{binormal1}
\chi_t=cb,
\end{equation*}
where $c$ is the curvature of the curve and $b$ its binormal. This geometric flow was proposed by Da Rios in 1906 \cite{DaR} as an approximation of the evolution 
of a vortex filament in a 3-D incompressible inviscid fluid. Simple explicit and relevant examples of solutions of \eqref{binormal} are the straight lines, that remain stationary, the circles, that move in the orthogonal
direction of the plane where they are contained and with velocity the inverse of the radius, and the helices that,  besides exhibiting the same rigid motion of the circles, rotate with a constant velocity around their axis as a corkskrew. We refer the reader to  \cite{AlKuOk}, \cite{Ba} and \cite{Sa} for an analysis and discussion about the limitations of this model and to \cite{Ri} for a survey about Da Rios' work. 

The selfsimilar solutions with respect to scaling of \eqref{binormal} are easily found by first fixing the ansatz 
\begin{equation}\label{ansatz}
\chi(t,x)=\sqrt{t}\,G\left(\frac {x}{\sqrt{t}}\right),
\end{equation}
and then solving the corresponding ordinary differential equation. In geometric terms the solutions are determined by a curve with the properties
\begin{equation*}\label{ctauselfsim}
c(x)=a,\quad\qquad\tau(x)=\frac{x}{2},
\end{equation*}
for a parameter $a>0$. Calling $G_a$ the corresponding curve and $T_a$ its unit tangent, it is rather easy to see that $T_a(x)$ has a limit $A_a^{\pm}$ 
as  $x$ goes to $\pm\infty$, so that $G_a$ approaches asymptoticaly to two lines. In the neighborhood of $x=0$ the curve is similar to a circle of radius $1/a$ and for large $s$ the curve has a helical shape of increasing pitch. Notice that equation \eqref{binormal} is reversible
in  time. So if at time $t=1$ the filament is given by $\chi_a(1,x)=G_a(x)$ the evolution $\chi_a(t,x)$ for $0<t<1$  is given by \eqref{ansatz}. From this expression we see that the two lines at infinity remain fixed. However,  the helices transport the ``energy"  from infinity towards the origin so that the overall effect is an increasing of the curvature, that becomes $a/\sqrt t$. The final configuration at time $t=0$ is given by the two lines determined by $A_a^{\pm}$. That these two lines are different is not so straightforward. It was proved in \cite{SJL} that
$$\sin\frac\theta 2=e^{-\frac{a^2}{2}},$$
where $\theta$ is the angle between the vectors $A^+_a$ and $-A^-_a$. As a consequence starting with $G_a$, a real analytic curve at $t=1$, a corner is created at time $t=0$. This particular solution is studied numerically in \cite{DGV}. One of the conclusions of that paper is that the process of concentration around the origin is very stable. Moreover the similarity between the numerical solutions and those that appear experimentally in a colored fluid traversing a delta wing is quite remarkable, see figure 1.1 in  \cite{DGV}.\\

The stability results proved in \cite{BV} are based on a tranformation due to Hasimoto  \cite{Ha}. He defines the so-called ``filament function" $\psi$ of  a regular solution of \eqref{binormal} that has strictly  positive curvature at all points. The precise expression is given by 
\begin{equation*}\label{Hasimoto}
\psi(t,x)=c(t,x)\exp{\left\{i\int\limits_0^x \tau(t,x')dx'\right\}}.
\end{equation*}
Then it is proved in \cite{Ha} that $\psi$ solves the nonlinear Schr\"odinger equation  
\begin{equation}\label{+Hasimoto-eq}
i\psi_t+\psi_{xx}+\frac{1}{2}\left(|\psi|^2-A(t)\right)\psi=0,
\end{equation}
with
\begin{equation*}\label{Hasimoto-coeff}
A(t)=\left(\pm2\frac{c_{xx}-c\,\tau^2}{c}+c^2\right)(t,0).
\end{equation*}
Notice that in \eqref{+Hasimoto-eq}, the non-linear term appears with the focusing sign. The opposite case, the defocusing one, can be obtained in a similar way by assuming that the tangent vector $\chi_s$ has a constant hyperbolic length instead of the constant euclidean length as in \eqref{binormal}. This equation has to be changed accordingly, see \cite{BV} and \cite{Pa} for the details.

The particular selfsimilar solution $\chi_a(t,x)$ of \eqref{binormal} has as curvature and torsion
$$c_a(t,x)=\frac{a}{\sqrt{t}},\quad\qquad\tau_a(t,x)=\frac{x}{2t},$$
so its filament function is
$$\psi_a(t,x)=a\frac{e^{i\frac{x^2}{4t}}}{\sqrt{t}}.$$
This function is a solution of \eqref{+Hasimoto-eq} if 
$$A(t)=\frac{a^2}{t}.$$
Notice that neither $\psi_a(t)$ nor any of its derivatives are in $L^2$, and that $\psi_a(0)=ae^{i\frac\pi 4}\delta_{x=0}$. This is a too singular initial data for the available theory (\cite{VaVe}, \cite{Gr}, \cite{Ch}, \cite{BV0}). 
Therefore one might think that this particular solution is not related to any natural energy. However, this is not the case, as can be proved by considering the pseudo-conformal transformation. Given $\psi$ solution of\footnote{For sake of simplicity we omit the $1/2$ factor in \eqref{+Hasimoto-eq}, that can be resorbed by a scaling argument.} 
\begin{equation}\label{Hasimoto-eq}
i\psi_t+\psi_{xx}\pm\left(|\psi|^2-\frac{a^2}{t}\right)\psi=0,
\end{equation}
we define a new unknown $v$ as
\begin{equation}\label{calT}
\psi(t,x)=\mathcal {T}v(t,x)=\D\frac{e^{i\frac{x^2}{4t}}}{\sqrt{t}}\overline{v}\left(\frac 1t,\frac xt\right).
\end{equation}
Then $\,v\,$ solves 
\begin{equation}\label{GP1}
iv_t+v_{xx}\pm\D\frac 1t\left(|v|^2-a^2\right)v=0,\\
\end{equation}
and $v_a=a$ is a particular solution corresponding to $\psi_a$. A natural quantity associated to \eqref{GP1} is the normalized energy
\begin{equation*}\label{energy}E(v)(t)=\frac{1}{2}\int |v_x(t)|^2\,dx\mp\frac{1}{4t}\int(|v(t)|^2-a^2)^2\,dx.
\end{equation*}
An immediate calculation gives that
\begin{equation*}\label{GP2}\partial_{t}E(v)(t)\mp\frac{1}{4t^2}\int(|v|^2-a^2)^2\,dx=0,
\end{equation*}
and in particular $E(v_a)=0.$\\

The first stability result we give in \cite{BV} is the proof of the existence for small $a$ of a modified wave operator for
solutions of \eqref{Hasimoto-eq}  that at time $t=1$ are close to the constant $v_a=a.$ Namely, we prove that if we fix an asymptotic
state $u_+$ small in $L^1\cap L^2$ there is a a unique solution of \eqref{Hasimoto-eq} for $t>1$ that behaves as time approaches infinity as
\begin{equation*}\label{1.14}
v_1(t,x)=a+e^{\pm i a^2\log t}e^{it\partial_x^2}u_+(x).
\end{equation*}
Here $e^{it\partial^2_x}$ denotes the free propagator. Therefore the free dynamics has to be modified by the long-range factor $e^{\pm i a^2\log t}$, due to the non-integrability of the coefficient $1/t$ that appears in \eqref{GP1}. This is similar to the framework of long range wave operators for cubic 1-d NLS (\cite{Ozawa},\cite{Ca},\cite{HaNa}). Here the situation is different since the $L^\infty$-norm of the functions we are working with is not decaying as $t$ goes to infinity, being just bounded. A link could also be made with the asymptotic results for the Gross-Pitaevskii equation around the constant solution (\cite{GNT}, \cite{GNT2}), but still our situation is not the same, and we treat the linearized equation in a different way. 

The condition $u_+\in L^1$ will be relaxed in this article to the weaker one that $\hat u_+(\xi)$ times positive powers of $|\xi|$ is bounded in a neighborhood of the origin. As we shall see, this latter assumption is the one  that naturally appears for proving the asymptotic completeness of \eqref{GP1}. Moreover, we shall prove in Theorem A.1 of Appendix A the existence of the modified wave operator by assuming this weaker property.

Once the solution $v$ is constructed we recover $\psi$ from \eqref{calT}. The result proved in \cite{BV} is that given $u_+$ as before, there exists a unique solution $\psi(t,x)$ of \eqref{Hasimoto-eq} such that  $\psi$ behaves as $\psi_1$ as $t$ goes to zero, with
$$\psi_1(t,x)=a\frac{e^{i\frac{x^2}{4t}}}{\sqrt{t}}+\frac{e^{\pm ia^2\log t}}{\sqrt{4\pi i}}\hat {\overline{u}}_+\left(-\frac x2\right),$$
The precise statement about the behavior of $\psi-\psi_1$ can be found in Corollary 1.2 of \cite{BV}. However, it is important to point out two facts. Firstly, the rate of convergence is $\|\psi-\psi_1\|_{L^2}<Ct^{\frac 14}$. And secondly, that although the singular term $a\frac {e^{i\frac{x^2}{4t}}}{\sqrt t}$ has a limit, the correction does not. As a consequence neither $\psi_1$ nor $\psi$ have a trace at $t=0$, no matter how good $u_+$ is. Notice also that the condition about the boundedness of $\hat u_+$ is understood here as that the perturbation of the singular solution $\psi_a$ has to be bounded close to the point where the singularity is created. 

The next result in \cite{BV} is the construction of solutions of \eqref{binormal} that are close to $\chi_a$. This is done by integrating the Frenet system using the filament function given by $\psi$. The role played by the euclidean geometry is crucial at this step, because by construction the binormal vector has unit euclidean length. Therefore to conclude the existence of  a trace for $\chi(t)$ at $t=0$ it is enough that the curvature, given by $|\psi(t,x)|$, is integrable at time zero. Although this is obtained by quite general $u_+$, even though there is not a trace for $\psi$ at $t=0$ as we already said, the question of the existence of a corner is much more delicate. In order to get it, it is necessary to improve the rate of convergence of $\psi-\psi_1$. This is done by assuming that $|\xi|^{-2}\hat u_+(\xi)$ is locally in $L^2$, see Theorem 1.5 in \cite{BV}.\\

Our main result in this paper is to prove the asymptotic completeness for solutions of \eqref{GP1} that at time $t=1$ are close to the constant $a$. In order to give the precise statement we have to make several transformations of \eqref{GP1}. First of all we write
\begin{equation}\label{w}v=w +a,
\end{equation}
so that $w$ has to be a solution of
\begin{equation}\label{GP}
iw_t+w_{xx}=\mp\D\frac 1t\left(|a+w|^2-a^2\right)(a+w).\\
\end{equation}
The right hand side of the above equation has two linear terms. One is $\mp\frac{a^2}{t}w$ that is resonant, and it is the one that creates the logarithmic correction of the phase . The other one is similar, but involves $\bar w$ and therefore it is not resonant. Then, we define $u$ as 
\begin{equation}\label{u}
u(t,x)=w(t,x) e^{\mp ia^2\log t}.
\end{equation}
As a consequence $u$ has to solve
$$iu_t=\left(iw_t\pm\frac{a^2}{t}w\right)e^{\mp ia^2\log t}=\left(-w_{xx}\mp \frac{|w|^2w+a(w^2+2|w|^2)}{t}\mp\frac{a^2}{t}\overline{w}\right)e^{\mp ia^2\log t},$$
so 
\begin{equation}\label{nonlin}
iu_t+u_{xx}\pm\frac{a^2}{t^{1\pm 2ia^2}}\overline{u}+\frac{F(u)}{t}=0,
\end{equation}
with $F(u)$ given by
\begin{equation}\label{F}
F(u)= F (w e^{\mp ia^2\log t})=\pm\frac{|w|^2w+a(w^2+2|w|^2)}{t}e^{\mp ia^2\log t}.
\end{equation}
As we see $F$ involves just quadratic and cubic terms of $u$.

Also, we need to introduce some auxiliary function spaces.  For fixed $\gamma$ and $t_0$ we define the space ${X_{t_0}^\gamma}$ of functions $f(x)$ such that the norm
\begin{equation}\label{Xtau}
\|f\|_{X_{t_0}^\gamma}=\frac{1}{t_0^\frac 14}\|f\|_{L^2}+\frac{t_0^\gamma}{\sqrt{t_0}}\||\xi|^{2\gamma}\hat{f}(\xi)\|_{L^\infty(\xi^2\leq 1)}
\end{equation}
is bounded, and ${Y_{t_0}^\gamma}$ the space of
functions $g(t,x)$ such that the norm
\begin{equation}\label{Ytau}
\|g\|_{Y_{t_0}^\gamma}=\sup_{t\geq t_0}\,\left(\frac{1}{t_0^\frac 14}\|g(t)\|_{L^2}+\left(\frac{t_0}{t}\right)^{a^2}\frac{t_0^\gamma}{\sqrt{t_0}}\||\xi|^{2\gamma}\hat{g}(t,\xi)\|_{L^\infty(\xi^2\leq 1)}\right)
\end{equation}
is finite.

We have the following result.
\begin{theorem}\label{theorem1}
Let $0\leq \gamma<\frac 14$, $0<a$ and let $u(1)$ be a function in $X_1^\gamma$ small with respect to $a$. Then there exists a unique global solution $u\in Z^\gamma=Y_1^\gamma\cap L^4((1,\infty),L^\infty)$ of equation \eqref{nonlin} with $u(1)$ initial data at time $t=1$, and
$$\|u\|_{Z^\gamma}\leq C(a)\, \|u(1)\|_{X_1^\gamma}.$$
Moreover, this solution scatters in $L^2$: there exists $f_+\in L^2$ for which
\begin{equation*}
\|u(t)-e^{i(t-1)\partial_x^2}f_+\|_{L^2}\leq \frac{C(a,\delta)}{t^{\frac14-(\gamma+\delta)}}\,\|u(1)\|_{X_1^\gamma}\underset{t\tend\infty}{\longrightarrow} 0,
\end{equation*}
for any $0<\delta<1/4-\gamma$. Finally, the asymptotic state  $f_+$ satisfies for all $\xi^2\leq 1$ the estimate
$$|\xi|^{2(\gamma+\delta)}|\hat{f_+}(\xi)|\leq C(a,\delta)\,\|u(1)\|_{X_1^\gamma}.$$
\end{theorem}

To obtain the theorem, we first study the linearized equation 
\begin{equation}\label{lin}
iu_t+u_{xx}\pm \frac{a^2}{t^{1\pm 2ia^2}}\overline{u}=0,\\
\end{equation}
with initial data $u(t_0,x)$ at time $t_0\geq 1$. We prove that $u(t)$ behaves for large times like a free Schr\"odinger evolution. The only difference is that the Fourier zero-mode of $u(t)$ can become singular. Then, by perturbative methods, we deduce the asymptotic completeness for the nonlinear equation \eqref{nonlin}. The main part of our proof uses Fourier analysis and exploits particularly the non-resonant structure of $\overline{u}$ in \eqref{lin}. This is done by oscillatory integral techniques and simple integration by parts arguments (see in particular Lemma \ref{IBP} below).

 As we see, even if at time $t=1$ we are assuming that $\hat u(1)$ remains bounded in a neighborhood of the origin, we cannot prove a similar property for the asymptotic state $f_+$. This is not just a technical question. In Appendix B2 we shall prove that if $xu(1)$ is in $L^2$, so that
 \begin{equation*}\label{phi}
 \phi(t)=\int^{\infty}_{-\infty}\, u(t,x)\,dx
 \end{equation*}
 is well defined for all $t>1$, then under some conditions on $u(1)$,
 $$|\phi(t)|\geq C \log t.$$
 This property is rather easy to obtain, at least at a formal level, for the linearized equation
 \begin{equation}\label{linearized}
iw_t+w_{xx}=\mp\frac{a^2}{t}(w+\bar w).\\
\end{equation}
In fact, call $y(t)=\Re\int^{\infty}_{-\infty}\, w(t,x)\,dx$ and $z(t)=\Im\int^{\infty}_{-\infty}\, w(t,x)\,dx$, then
$$ i y'(t)-z'(t)=\mp2\frac{a^2}{t}y(t).$$
Hence $y(t)=y(1)$ and $z(t)=z(1)\pm2a^2y(1)\log t$.\\

Our next step is to understand the above result in terms of the filament function $\psi(t,x)$. From \eqref{calT},  \eqref{w}, and  \eqref{u} we have for $0<t\leq 1$
 \begin{equation}\label{1.30}
 \psi(t,x)=a\frac{e^{i\frac{x^2}{4t}}}{\sqrt {t}}+ e^{\pm ia^2\log t}\,\mathcal {T} u(t,x).
 \end{equation}
Therefore
$$\psi(1,x)=ae^{ix^2}+\psi_1(x),$$
with $\psi_1(x)=e^{ix^2}u(1,x)$. For simplicity we will impose $\psi_1\in L^1\cap L^2$ to fullfil the hypothesis $|\xi|^{2\gamma}\widehat{u}(1,\xi)\in L^\infty(\xi^2\leq 1)\cap L^2$ needed in Theorem \ref{theorem1} with $\gamma=0$. Then it will follow the existence of an $f_+\in L^2$ such that $u(t)$ behaves like $e^{i(t-1)\partial_x^2}f_+$. Now, on the one hand, the pseudo-conformal transform of $e^{i(t-1)\partial_x^2}f_+$ is the free evolution of $\frac{1}{\sqrt{4\pi i}}\,\widehat{e^{i\partial_x^2}\,\overline {f_+}}\left(-\frac{\cdot}{2}\right)$.
On the other hand $\mathcal{T}$ is an isometry of $L^2$. As a consequence we obtain from Theorem \ref{theorem1} the following scattering result.

\begin{theorem}\label{theorem2}
Let $0<a$ and let $\psi_1$ be a small function in $L^1\cap L^2$ with respect to $a$. Then there exists a unique solution $\psi$ of equation \eqref{Hasimoto-eq} for $0<t\leq 1$ with
 \begin{equation*}\label{1.31}
 \psi(1,x)=ae^{i\frac{x^2}{4}}+ \psi_1(x),
 \end{equation*}
 such that $ \psi(t,x)-a\frac{e^{i\frac{x^2}{4t}}}{\sqrt{t}}\in L^\infty((0,1),L^2)\cap  L^4((0,1),L^\infty).$
 Moreover, there exists $\psi_+\in L^2$ such that 
 \begin{equation*}\label{1.32}
 \left\| \psi(t,x)-a\frac{e^{i\frac{x^2}{4t}}}{\sqrt{t}}-e^{\pm ia^2\log t}e^{it\partial_x^2}\psi_+(x)\right\|_{L^2} \leq C(a,\delta)\,t^{\frac14-\delta}\|\psi_1\|_{L^1\cap L^2},
\end{equation*}
for any $0<\delta<1/4$,
and for $|x|\leq 2$ 
 \begin{equation*}\label{1.33}
|x|^{2\delta}|\psi_+(x)| \leq C(a,\delta)\,\|\psi_1\|_{L^1\cap L^2}.
 \end{equation*}
\end{theorem}

As we shall see in Corollary \ref{nonlinearascHs}, if $u_1$ is regular in terms of Sobolev spaces so is the solution $u(t)$ given in Theorem \ref{theorem1}. So in particular $u(t)$ is uniformly bounded in terms of the size of $u_1$. Then from \eqref{1.30} we conclude that if $u_1$ is small enough with respect to $a$ then $\frac{a}{2\sqrt{t}}\leq |\psi (t,x)|\leq \frac{3a}{2\sqrt{t}}$, and therefore $ |\psi (t,x)|$ becomes singular as $t$ goes to zero. Hence we can use the Frenet system to construct $\chi(t,x)$ a regular solution of \eqref{binormal} for $0<t\leq 1$, and the corresponding Frenet frame, that will become also singular as $t$ approaches to zero (see for instance \cite{NaShVeZe} or the Appendix of \cite{BV}). Notice also that this argument works in both settings, focusing and defocusing.   
Moreover, due to the fact that in the focusing situation the binormal has unit euclidean length, and that the curvature is integrable in time, we can define $\chi_0(x)$ as
\begin{equation}\label{chi0}
\chi_0(x)=\chi(1,x)-\int_0^1 c(\tau,x)\,b(\tau,x)\,d\tau.
\end{equation}

As a conclusion we have the following result. 
\begin{theorem}\label{theorem3}
Let $0<a$ and $\chi_1(x)$ a regular curve with curvature and torsion $c_1$ and $\tau_1$. We define
$$\psi_1(x)=c_1(x)e^{i\int_0^x\tau_1(x')\,dx'},\quad u_1(x)=e^{-i\frac{x^2}{4}}\psi_1(x)-a,$$
and assume that $u_1\in L^1\cap H^3$ small with respect to $a$. Then there exists a unique $\chi(t,x)$ regular solution of \eqref{binormal} for $0<t\leq 1$ with $\chi(1,x)=\chi_1(x)$. Moreover, its curvature and torsion $c$ and $\tau$ satisfy
\begin{equation}\label{ctauest}
\left|c (t,x)-\frac{a}{\sqrt{t}}\right|\leq \frac{C(u_1)}{t^{\frac 14^+}},\quad  \left|\tau(t,x)-\frac{x}{2t}\right|\leq \frac{C(u_1)}{t^{\frac 34^+}},
\end{equation}
and by defining $\chi_0(x)$ as in \eqref{chi0} then
$$|\chi(t,x)-\chi_0(x)|\leq C(u_1)\sqrt{t}.$$
\end{theorem}

\begin{remark}
The bounds of the curvature and torsion given in \eqref{ctauest} follow from their definition
$$c(t,x)=|\psi(t,x)|, \quad \tau(t,x)=\Im\frac{\partial_x \psi(t,x)}{\psi(t,x),}$$
and from the rate of decay obtained in Corollary \ref{nonlinearascHs} below. The same calculations can be found in \S 3.2 of \cite{BV}, therefore they will be omitted here.
\end{remark}

\begin{remark}
As we said before, by Theorem 1.5 in \cite{BV}, if $a$ is small enough and if $\psi_+$ is small and regular enough with $|x|^{-2} \psi_+$ locally integrable, then $\chi_0(x)$ has a corner at the origin $x=0$.  
\end{remark}
\begin{remark}
The use of the Frenet frame can be avoided. In fact, once a solution of \eqref{Hasimoto-eq} is obtained, a slight modification of Theorem 3.1 of \cite{NaShVeZe} can be used to construct a solution for \eqref{binormal} for $0<t\leq 1$, with a trace $\chi_0$ in the focusing case defined as in \eqref{chi0}. This is because $|\psi|^2-\frac{a^2}{t}$ is in $L^2((\epsilon,1),L^\infty)$ for any positive $\epsilon$. In this case $|\psi|$ becomes unbounded in the Strichartz norm $L^4((0,1),L^\infty)$, and therefore the corresponding frame will become also singular as $t$ approaches to zero, as does the Frenet frame. 
\end{remark}

The paper is organized as follows. In Section \S\ref{s:linear} we study the asymptotic completeness of the linear equation 
\eqref{lin}. Then in Section \S\ref{s:nonlinear} we deduce Theorem \ref{theorem1} by perturbative methods. As already mentioned, Appendix A contains the proof of a new version of the existence of the wave operator
of \eqref{nonlin} that fits better with the hypothesis needed to obtain the asymptotic completeness of Theorem \ref{theorem1}. Finally in Appendix B we prove the growth of the zero Fourier mode for the solutions of the linear and the non-linear equations, \eqref{lin} and \eqref{nonlin}, property that we think it is interesting in itself.\\

{\bf{Acknowledgements:}} The authors are grateful to Kenji Nakanishi for useful remarks\\ concerning Lemma \ref{controls2}. We also want to thank the referee for the careful reading and comments that highly improved the presentation of the paper.

First author was partially supported by the French ANR projects:
ANR-05-JCJC-0036, ANR-05-JCJC-51279 and  R.A.S. ANR-08-JCJC-0124-01. The second author was partially supported by the grant MTM 2007-62186 of MEC (Spain) and FEDER.

\section{Scattering for the linear equation}\label{s:linear}
In this section we consider only the linear equation \eqref{lin}:
$$iu_t+u_{xx}\pm \frac{a^2}{t^{1\pm 2ia^2}}\overline{u}=0,$$
with initial data $u(t_0,x)$ at time $t_0\geq 1$. We start in \S\ref{ss-lin-controls} with the proof of some a-priori estimates on the Fourier modes of $u(t)$, that will allow us in \S \ref{ss-lin-global} to get a satisfactory global existence result. Then in \S \ref{ss-lin-asc} we prove the asymptotic completeness for \eqref{lin}, again with the help of the properties pointed out in \S\ref{ss-lin-controls}. Finally, in \S \ref{ss-lin-aposteriori} we obtain a regularity result for the asymptotic state and we prove a-posteriori  that $u\in L^4((t_0,\infty),L^\infty)$.

\subsection{A-priori controls}\label{ss-lin-controls}
\begin{lemma}\label{controls}
If $u$ solves equation \eqref{lin} then for $0<t_0\leq t$,
\begin{equation}\label{Fmod}
|\hat{u}(t,\xi)|\leq \,\frac{ t^{a^2}}{t_0^{a^2}} \,\left(|\hat{u}(t_0,\xi)|+|\hat{u}(t_0,-\xi)|\right).
\end{equation}
In particular,
$$\|u(t)\|_{\dot{H}^k}\leq\,\frac{ t^{a^2}}{t_0^{a^2}}\|u(t_0)\|_{\dot{H}^k}
$$
for all $k\in\mathbb{Z}$.
\end{lemma}
\begin{proof}
Using the Fourier transform we write equation \eqref{lin} as
\begin{equation}\label{Ft}
0=i\hat{u}_t(t,\xi)-\xi^2\hat{u}(t,\xi)\pm \frac{a^2}{t^{1\pm 2ia^2}}\hat{\overline{u}}(t,\xi)=i\hat{u}_t(t,\xi)-\xi^2\hat{u}(t,\xi)\pm \frac{a^2}{t^{1\pm 2ia^2}}\overline{\hat{u}(t,-\xi)}.
\end{equation}
By multiplying by $\overline{\hat{u}(t,\xi)}$ and by taking the imaginary part,
$$\partial_t|\hat{u}(t,\xi)|^2=\mp 2\Im\, \frac{a^2}{t^{1\pm 2ia^2}}\,\overline{\hat{u}(t,-\xi)}\,\overline{\hat{u}(t,\xi)}.$$
We obtain
$$\partial_t|\hat{u}(t,\xi)|\leq\frac{a^2}{t}\,|\hat{u}(t,-\xi)|,$$
therefore 
$$\partial_t\left(|\hat{u}(t,\xi)|+|\hat{u}(t,-\xi)|\right)\leq \frac{a^2}{t} \left(|\hat{u}(t,\xi)|+|\hat{u}(t,-\xi)|\right),$$
so the lemma follows.
\end{proof}

Now we shall improve this control for some small frequencies.

\begin{lemma}\label{controls2}
Let $0<\delta$. If $u$ solves equation \eqref{lin} then for all $\xi\neq 0$ and for all $0<t_0\leq t$,
\begin{equation}\label{Fmod2}
|\hat u(t,\xi)|\leq \left(C(a)+\frac{C(a,\delta)}{(\xi^2\,t_0)^\delta}\right)\left(|\hat u(t_0,\xi)|+|\hat u(t_0,-\xi)|\right),
\end{equation}
which is a better estimate than the one of Lemma \ref{controls} in the region $\frac{1}{t^{a^2}}\ls \xi^{2\delta}$.
\end{lemma}

\begin{proof}We shall work with $w(t)=u(t)e^{\pm ia^2\log t}$ the solution of \eqref{linearized}:
$$i\partial_tw+w_{xx}\pm\frac{a^2}{t}(w+\overline{w})=0.$$
We have, by taking the Fourier modes of the real and imaginary part of $w$,
\begin{equation}\label{remodes}
\partial_t\,\widehat{\Re w}(t,\xi)=\xi^2\,\widehat{\Im w}(t,\xi),
\end{equation}
\begin{equation}\label{immodes}
\partial_t\,\widehat{\Im w}(t,\xi)=-\xi^2\,\widehat{\Re w}(t,\xi)\pm \frac {2a^2}{t}\,\widehat{\Re w}(t,\xi).
\end{equation}

We denote 
$$Y_\xi( t)=\widehat{\Re w}\left(\frac{t}{\xi^2},\xi\right)\,\,,\,\,Z_\xi( t)=\widehat{\Im w}\left(\frac{t}{\xi^2},\xi\right).$$
Equations \eqref{remodes} and \eqref{immodes} become
\begin{equation}\label{system}
Y_\xi '( t)=Z_\xi( t)\,\,,\,\,Z_\xi '(t)=\frac{1}{\xi^2}\left(-\xi^2+\frac{2a^2\xi^2}{t}\right)\,Y_\xi(t)=\left(-1+\frac{2a^2}{ t}\right)\,Y_\xi(t).
\end{equation}
For simplicity, we consider only the focusing case, that is slightly more complicated. 
For $0<\epsilon\leq 1$ to be chosen later, the function
$$\sigma_\xi( t)= \frac 1\epsilon |Y_\xi( t)|^2+\epsilon |Z_\xi( t)|^2$$
satisfies
$$\sigma_\xi'=\left(\frac 1\epsilon +\epsilon\left(-1+\frac{2a^2}{ t}\right)\right)2\Re \overline{Y_\xi} Z_\xi\leq \left(\frac 1\epsilon -\epsilon+\epsilon\frac{2a^2}{t}\right)\sigma_\xi.$$
Therefore
$$\left(\log\sigma_\xi - t\left(\frac 1\epsilon -\epsilon\right)-2a^2\epsilon\log  t \right)'\leq 0,$$
and finally for all $0<\tilde t_0\leq t$,
$$\sigma_\xi( t)\leq e^{\Phi(t)}\sigma_\xi(\tilde t_0),$$
where
$$\Phi(t)=( t-\tilde t_0)\left(\frac 1\epsilon -\epsilon\right)+2a^2\epsilon(\log  t-\log\tilde t_0) .$$

\begin{itemize} 
\item Case 1: $0<\tilde t_0\leq t\leq \min\{a^2,\frac 1e\}$.\\
In this region
$$\sigma_\xi( t)\leq e^{\frac t\epsilon-2a^2\epsilon\log\tilde t_0}\sigma_\xi(\tilde t_0)\leq e^{\frac {a^2}{\epsilon}+2a^2\epsilon|\log\tilde t_0| }\sigma_\xi(\tilde t_0).$$ 
By choosing 
$$\epsilon =\frac{1}{\sqrt{|\log \tilde t_0|}},$$
we get
$$\sigma_\xi( t)\leq e^{3a^2\sqrt{|\log\tilde t_0| }}\sigma_\xi(\tilde t_0).$$ 
It follows that 
$$|Y_\xi(t)|^2\leq \left(|Y_\xi(\tilde t_0)|^2+\frac{|Z_\xi(\tilde t_0)|^2}{|\log\tilde t_0|} \right)e^{3a^2\sqrt{|\log\tilde t_0| }},$$
and
$$|Z_\xi(t)|^2\leq \left(|\log\tilde t_0||Y_\xi(t_0)|^2+|Z_\xi(\tilde t_0)|^2\right)e^{3a^2\sqrt{|\log\tilde t_0| }}.$$
Therefore, for all $\delta>0$, there exists a constant $C(a,\delta)$ such that for all $0<\tilde t_0\leq t\leq \min\{a^2,\frac 1e\}$,
$$|Y_\xi(t)|^2+|Z_\xi(t)|^2\leq\frac{C(a,\delta)}{\tilde t_0^{2\delta}}(|Y_\xi(\tilde t_0)|^2+|Z_\xi(\tilde t_0)|^2).$$

\item Case 2: $\min\{a^2,\frac 1e\}\leq \tilde t_0\leq t\leq 4a^2$ (if such a situation exists).\\
In this case, by taking $\epsilon =1$, $\Phi(t)$ is bounded by a constant depending on $a$, and we get
$$|Y_\xi(t)|^2+|Z_\xi(t)|^2\leq C(a)(|Y_\xi(\tilde t_0)|^2+|Z_\xi(\tilde t_0)|^2).$$

\item Case 3: $4a^2< \tilde t_0\leq t$.\\
For this region we shall diagonalize the system
$$\partial_t\left(\begin{array}{c}Y_\xi\\ Z_\xi\end{array}\right)
=\left(\begin{array}{cc}0 &1\\-\left(1-\frac{2a^2}{t}\right) & 0\end{array}\right)\left(\begin{array}{c}Y_\xi\\ Z_\xi\end{array}\right).$$
Let 
$$\alpha(t)=\sqrt{1-\frac{2a^2}{t}}\quad,\quad 
P(t)=\left(\begin{array}{cc} 1&1\\i\alpha(t)&-i\alpha(t)\end{array}\right).$$
In particular,
$$\frac{1}{\sqrt{2}}\leq \alpha(t)\leq 1\quad,\quad P^{-1}(t)=\left(\begin{array}{cc} \frac 12&-\frac{i}{2\alpha(t)}\\\frac 12&\frac{i}{2\alpha(t)}\end{array}\right).$$
Then
$$\left(\begin{array}{c}\tilde Y_\xi(t)\\ \tilde Z_\xi(t)\end{array}\right)=P^{-1}(t)\left(\begin{array}{c}Y_\xi(t)\\ Z_\xi(t)\end{array}\right)$$
satisfies
$$\partial_t\left(\begin{array}{c}\tilde Y_\xi\\ \tilde Z_\xi\end{array}\right)=\partial_t(P^{-1})P\left(\begin{array}{c}\tilde Y_\xi\\ \tilde Z_\xi\end{array}\right)+\left(\begin{array}{cc} i\alpha&0\\0&-i\alpha\end{array}\right)\left(\begin{array}{c}\tilde Y_\xi\\ \tilde Z_\xi\end{array}\right).$$
Denote
$$\Phi(t)=t-a^2\log t-\int_t^\infty \alpha(s)-1+\frac{a^2}{s}\,ds.$$
Finally,
$$\left(\begin{array}{c}\mathring Y_\xi(t)\\ \mathring Z_\xi(t)\end{array}\right)=
\left(\begin{array}{cc} e^{-i\Phi(t)} &0\\0 & e^{i\Phi(t)}\end{array}\right)\left(\begin{array}{c}\tilde Y_\xi(t)\\ \tilde Z_\xi(t)\end{array}\right)$$
satisfies
\begin{equation}\label{scandinavian}
\partial_t\left(\begin{array}{c}\mathring Y_\xi\\ \mathring Z_\xi\end{array}\right)=M(t)\,\left(\begin{array}{c}\mathring Y_\xi\\ \mathring Z_\xi\end{array}\right),
\end{equation}
where
$$M(t)=\left(\begin{array}{cc} e^{-i\Phi(t)} &0\\0 & e^{i\Phi(t)}\end{array}\right)\partial_t(P^{-1})P\left(\begin{array}{cc} e^{i\Phi(t)} &0\\0 & e^{-i\Phi(t)}\end{array}\right)=\frac{a^2}{2t^2\alpha^2}
\left(\begin{array}{cc} -1 & e^{-2i\Phi(t)}\\e^{2i\Phi(t)} & -1\end{array}\right)$$
Since $\frac{1}{\sqrt{2}}\leq \alpha(t)\leq 1$,  all the entries of $M(t)$ are upper-bounded by $\frac {Ca^2}{t^2}$. We infer that
$$\partial_t(|\mathring Y_\xi|^2+|\mathring Z_\xi|^2)\leq \frac {Ca^2}{t^2}(|\mathring Y_\xi|^2+|\mathring Z_\xi|^2),$$
so
$$\partial_t\left(\log(|\mathring Y_\xi|^2+|\mathring Z_\xi|^2)+\frac {Ca^2}{t}\right)\leq 0.$$
We have $\frac {Ca^2}{\tilde t_0}\leq \frac C4$, and we get
$$|\mathring Y_\xi(t)|^2+|\mathring Z_\xi(t)|^2\leq C(|\mathring Y_\xi(\tilde t_0)|^2+|\mathring Z_\xi(\tilde t_0)|^2).$$
Finally, from the relation
$$|\mathring Y_\xi(t)|^2+|\mathring Z_\xi(t)|^2=\left|\frac 12 Y_\xi-\frac{i}{2\alpha}Z_\xi\right|^2+\left|\frac 12 Y_\xi+\frac{i}{2\alpha}Z_\xi\right|^2= \frac{1}{2}|Y_\xi|^2+\frac{1}{2\alpha^2}|Z_\xi|^2,$$
and from $\frac{1}{\sqrt{2}}\leq \alpha(t)\leq 1$ it follows that
\begin{equation}\label{boundsmodes}
|Y_\xi(t)|^2+|Z_\xi(t)|^2\leq C(|Y_\xi(\tilde t_0)|^2+|Z_\xi(\tilde t_0)|^2).
\end{equation}

\end{itemize}

Summarizing, we have obtained that for all $\delta>0$, 
there exists a constant $C(a,\delta)$ such that for all $0<\tilde t_0\leq t$,
\begin{equation}\label{boundsmodesbis}
|Y_\xi(t)|^2+|Z_\xi(t)|^2\leq\left(C(a)+\frac{C(a,\delta)}{\tilde t_0^{2\delta}}\right)(|Y_\xi(\tilde t_0)|^2+|Z_\xi(\tilde t_0)|^2).
\end{equation}

By recovering the first unknowns, for all $0<t_0\leq t$,
$$ |\widehat{\Re w}(t,\xi)|^2+|\widehat{\Im w}(t,\xi)|^2\leq \left(C(a)+\frac{C(a,\delta)}{(\xi^2\,t_0)^{2\delta}}\right) \left(|\widehat{\Re w}(t_0,\xi)|^2+|\widehat{\Im w}(t_0,\xi)|^2\right),$$
and by using the identity $2(|z_1|^2+|z_2|^2)=|z_1+iz_2|^2+|z_1-iz_2|^2$,
$$|\widehat{w}(t,\xi)|^2+|\widehat{w}(t,-\xi)|^2\leq\left(C(a)+\frac{C(a,\delta)}{(\xi^2\,t_0)^{2\delta}}\right)\left(|\hat w(t_0,\xi)|^2+|\hat w(t_0,-\xi)|^2\right).$$
Since $w(t)=u(t)e^{\pm ia^2\log t}$ the Lemma follows.\\

For further use we want to compute the asymptotic behavior of the solution $u$ of \eqref{lin}. In view of \eqref{scandinavian} and \eqref{boundsmodes} of Case 3, we can define for $4a^2\leq \tilde{t_0}$
$$\left(\begin{array}{c}\mathring Y^+_\xi\\ \mathring Z^+_\xi\end{array}\right)=\left(\begin{array}{c}\mathring Y_\xi(\tilde{t_0})\\ \mathring Z_\xi(\tilde{t_0})\end{array}\right)+\int_{\tilde{t_0}}^\infty M(\tau)\,\left(\begin{array}{c}\mathring Y_\xi(\tau)\\ \mathring Z_\xi(\tau)\end{array}\right)\,d\tau,$$
so that for $4a^2\leq \tilde{t_0}\leq t$
\begin{equation}\label{asmodes}
\left(\begin{array}{c}\mathring Y^+_\xi\\ \mathring Z^+_\xi\end{array}\right)=\left(\begin{array}{c}\mathring Y_\xi(t)\\ \mathring Z_\xi(t)\end{array}\right)+\int_{t}^\infty M(\tau)\,\left(\begin{array}{c}\mathring Y_\xi(\tau)\\ \mathring Z_\xi(\tau)\end{array}\right)\,d\tau,
\end{equation}
and
\begin{equation}\label{asmodes3}
|\mathring Y_\xi(t)-\mathring Y^+_\xi|+|\mathring Z_\xi(t)-\mathring Z^+_\xi|\leq \frac{C(a)}{t}(|Y_\xi(\tilde t_0)|+|Z_\xi(\tilde t_0)|).
\end{equation}
We have
$$\mathring Y^+_\xi=\mathring Y_\xi (t)+\int_{t}^\infty \frac{a^2}{2\tau^2\alpha^2}\left(-\mathring Y_\xi(\tau)+e^{-2i\Phi(\tau)}\mathring Z_\xi(\tau)\right)\,d\tau$$
$$=e^{-i\Phi(t)}\tilde Y_\xi (t)+\int_{t}^\infty \frac{a^2\, e^{-i\Phi(\tau)}}{2\tau^2\alpha^2}\left(-\tilde Y_\xi(\tau)+\tilde Z_\xi(\tau)\right)\,d\tau$$
$$=e^{-i\Phi(t)}\left(\frac{1}{2}Y_\xi (t)-\frac{i}{2\alpha}Z_\xi (t)\right)+\int_{t}^\infty \frac{a^2\, e^{-i\Phi(\tau)}}{2\tau^2\alpha^2}\frac{i}{\alpha}Z_\xi(\tau)\,d\tau,$$
and
$$\mathring Z^+_\xi =\mathring Z_\xi (t)+\int_{t}^\infty \frac{a^2}{2\tau^2\alpha^2}\left(e^{2i\Phi(\tau)}\mathring Y_\xi(\tau)-\mathring Z_\xi(\tau)\right)\,d\tau$$
$$=e^{i\Phi(t)}\tilde Z_\xi (t)+\int_{t}^\infty \frac{a^2\,e^{i\Phi(\tau)}}{2\tau^2\alpha^2}\left(\tilde Y_\xi(\tau)-\tilde Z_\xi(\tau)\right)\,d\tau$$
$$=e^{i\Phi(t)}\left(\frac{1}{2}Y_\xi (t)+\frac{i}{2\alpha}Z_\xi (t)\right)-\int_{t}^\infty \frac{a^2\,e^{i\Phi(\tau)}}{2\tau^2\alpha^2}\frac{i}{\alpha}Z_\xi(\tau)\,d\tau,$$
therefore since $\overline{Y_\xi}=Y_{-\xi}$ and $\overline{Z_\xi}=Z_{-\xi}$ we get the relation
\begin{equation}\label{asmodes4}
\overline{\mathring Y^+_\xi}=e^{i\Phi(t)}\left(\frac{1}{2}Y_{-\xi} (t)+\frac{i}{2\alpha}Z_{-\xi} (t)\right)-\int_{t}^\infty \frac{a^2\, e^{i\Phi(\tau)}}{2\tau^2\alpha^2}\frac{i}{\alpha}Z_{-\xi}(\tau)\,d\tau=\mathring Z^+_{-\xi}. 
\end{equation}

As a conclusion, by \eqref{asmodes3} and \eqref{boundsmodesbis} we get for all $0<\tilde t_0$ and all $t\geq\max\{\tilde t_0, 4a^2\}$,
\begin{equation}\label{asmodes2}
\left|\left(\frac 12 Y_\xi-\frac{i}{2\alpha}Z_\xi\right)-e^{i\Phi(t)}\mathring Y^+_\xi\right|+\left|\left(\frac 12 Y_\xi+\frac{i}{2\alpha}Z_\xi\right)-e^{-i\Phi(t)}\mathring Z^+_\xi \right|\end{equation}
$$=\left|\left(\frac 12 Y_{-\xi}+\frac{i}{2\alpha}Z_{-\xi}\right)-e^{-i\Phi(t)}\mathring Z^+_{-\xi} \right|+\left|\left(\frac 12 Y_\xi+\frac{i}{2\alpha}Z_\xi\right)-e^{-i\Phi(t)}\mathring Z^+_\xi \right|$$
$$\leq \frac 1t\left(C(a)+\frac{C(a,\delta)}{\tilde{t_0}^\delta}\right)(|Y_\xi(\tilde t_0)|+|Z_\xi(\tilde t_0)|).
$$
In particular, in view of the definition of $\alpha(t)$ and of estimate \eqref{boundsmodes}, we have
$$\left|\left(\frac 12 Y_\xi+\frac{i}{2}Z_\xi\right)-e^{-i\Phi(t)}\mathring Z^+_\xi \right|\leq \frac 1t\left(C(a)+\frac{C(a,\delta)}{\tilde{t_0}^\delta}\right)(|Y_\xi(\tilde t_0)|+|Z_\xi(\tilde t_0)|).
$$

Hence noticing that $\Phi(t)=t-a^2\log t+\mathcal{O}\left(\frac 1t\right)$ we get that $u_+$ defined by
\begin{equation}\label{defu+}
2\mathring Z^+_\xi=e^{-ia^2\log\xi^2}\hat{u_+}(\xi),
\end{equation}
satisfies for all $0<t_0$ and for all $t\geq\max\{ t_0, \frac{4a^2}{\xi^2}\}$ the estimate
\begin{equation}\label{u+modes}
|\hat u(t,\xi)-e^{-it\xi^2}\,\hat{u_+}(\xi)|\leq \frac {1}{\xi^2\,t}\left(C(a)+\frac{C(a,\delta)}{(\xi^2\,t_0)^\delta}\right)\left(|\hat u(t_0,\xi)|+|\hat u(t_0,-\xi)|\right).
\end{equation}
By combining this estimate with \eqref{Fmod2} for time $t=\frac{4a^2}{\xi^2}$ and for time  $0<t_0\leq t\leq \frac{4a^2}{\xi^2}$, we see that \eqref{u+modes} is valid for all $0<t_0\leq t$.

\end{proof}

\begin{remark}
Let us notice that the logarithmic loss is generally unavoidable. Suppose $Y_\xi(\tilde t_0)= Z_\xi(\tilde t_0)=1$ and $0<\tilde t_0\leq t\leq \min\{a^2,\frac 1e\}$. Then in view of the system \eqref{system}, we have that $Y_\xi(t)>1$ and $Z_\xi(t)>1$, and so
$$ Y_\xi(t)>Y_\xi(\tilde t_0),\quad Z'_\xi(t)>\left(-1+\frac{2a^2}{t}\right)Y_\xi(\tilde t_0)=-1+\frac{2a^2}{t}.$$
Then we get finally the logarithmic lower bound
$$Z_\xi(t)\geq Z_\xi(\tilde t_0)-2a^2\log\frac{t}{\tilde t_0}-(t-\tilde t_0)\geq C(a) |\log\tilde t_0|.$$
\end{remark}

\begin{remark}
In \S\ref{ss:appendix-fourier-lin} we shall see that if $\hat u(t_0,0)$ is defined and if $\hat u(t_0,0)\neq 0$, then also for $\xi=0$ a logarithmic loss is unavoidable, independently of the size of $t_0\leq t$:
\begin{equation}\label{evolutionzeromodes}
\hat u(t,0)=e^{\pm ia^2\log \frac{t_0}{t}}\hat u(t_0,0)\pm2ia^2e^{\pm ia^2\log \frac{t_0}{t}}\,\Re\hat u(t_0,0) \,\log\frac{t}{t_0}.
\end{equation}
Moreover, under certain conditions on the initial data, a logarithmic loss will be shown in \S \ref{ss:appendix-fourier-nonlin} for the zero-modes of the solutions of the nonlinear equation \eqref{nonlin}.
\end{remark}

We end this subsection with an estimate on the typical Duhamel term associated to \eqref{lin}.
\begin{lemma}\label{IBP}Let $0<\delta$. 
Let $u$ be a solution of equation \eqref{lin} and let
$$A_{t_1,t_2}(\xi)=a^2\int_{t_1}^{t_2}e^{-i(t-\tau)\xi^2}\,\frac{\overline{\hat{u}(\tau,-\xi)}}{\tau^{1\pm 2ia^2}}d\tau$$
be the Fourier transform of the Duhamel term integrated between two arbitrary times $t_0<t_1\leq t_2$.
Then for $\xi\neq 0$
\begin{equation}\label{IBPest}
|A_{t_1,t_2}(\xi)|\leq \left(C(a)+\frac{C(a,\delta)}{(\xi^2\,t_0)^\delta}\right) \frac{|\hat{u}(t_0,\xi)|+|\hat{u}(t_0,-\xi)|}{\xi^2\,t_1}.
\end{equation}
\end{lemma}

\begin{proof}
We perform an integration by parts
$$A_{t_1,t_2}(\xi)=a^2e^{-it\xi^2}\int_{t_1}^{t_2}\frac{\partial_\tau e^{i\tau\xi^2}}{i\xi^2}\,\frac{\overline{\hat{u}(\tau,-\xi)}}{\tau^{1\pm 2ia^2}}d\tau$$
$$=\left.\frac{a^2e^{-i(t-\tau)\xi^2}}{i\xi^2\,\tau^{1\pm 2ia^2}}\overline{\hat{u}(\tau,-\xi)}\right|_{t_1}^{t_2}-a^2\int_{t_1}^{t_2}\frac{e^{-i(t-\tau)\xi^2}}{i\xi^2}\,\frac{\partial_\tau\overline{\hat{u}(\tau,-\xi)}}{\tau^{1\pm 2ia^2}}-\frac{(1\pm2ia^2)e^{-i(t-\tau)\xi^2}}{i\xi^2\,\tau^{2\pm 2ia^2}}\overline{\hat{u}(\tau,-\xi)}\,d\tau.$$
From \eqref{Ft} we get
$$i\hat{u}_t(t,-\xi)-\xi^2\hat{u}(t,-\xi)\pm \frac{a^2}{t^{1\pm 2ia^2}}\overline{\hat{u}(t,\xi)}=0,
$$
and then
$$-i\overline{\hat{u}_t(t,-\xi)}-\xi^2\overline{\hat{u}(t,-\xi)}\pm \frac{a^2}{t^{1\mp 2ia^2}}\hat{u}(t,\xi)=0.$$
Therefore by replacing
$$\partial_\tau\overline{\hat{u}(\tau,-\xi)}=i\xi^2\overline{\hat{u}(\tau,-\xi)}\mp i \frac{a^2}{\tau^{1\mp 2ia^2}}\hat{u}(\tau,\xi)$$
we recover an $A_{t_1,t_2}(\xi)$ with sign minus, so that
$$A_{t_1,t_2}(\xi)=\left.\frac{a^2e^{-i(t-\tau)\xi^2}}{2i\xi^2\,\tau^{1\pm 2ia^2}}\overline{\hat{u}(\tau,-\xi)}\right|_{t_1}^{t_2}$$
$$-a^2\int_{t_1}^{t_2}\frac{e^{-i(t-\tau)\xi^2}}{2i\xi^2}\,\frac{\mp ia^2\hat{u}(\tau,\xi)}{\tau^{2}}-\frac{(1\pm2ia^2)e^{-i(t-\tau)\xi^2}}{2i\xi^2\,\tau^{2\pm 2ia^2}}\overline{\hat{u}(\tau,-\xi)}\,d\tau.
$$
Then we can upper-bound
$$|A_{t_1,t_2}(\xi)|\leq \frac{a^2}{2\xi^2\,t_2}|\hat{u}(t_2,-\xi)|+\frac{a^2}{2\xi^2\,t_1}|\hat{u}(t_1,-\xi)|$$
$$+\frac{a^2}{2\xi^2}\int_{t_1}^{t_2}\left(a^2|\hat{u}(\tau,\xi)|+|1+2ia^2||\hat{u}(\tau,-\xi)|\right)\,\frac{d\tau}{\tau^2}.$$
Now Lemma \ref{controls2} allows us to conclude, 
$$|A_{t_1,t_2}(\xi)|\leq 
a^2(a^2+|1+2ia^2|)\left(C(a)+\frac{C(a,\delta)}{(\xi^2\,t_0)^\delta}\right)\frac{|\hat{u}(t_0,\xi)|+|\hat{u}(t_0,-\xi)|}{2\xi^2\,t_1}$$
and the Lemma follows.
\end{proof}

\subsection{Global solutions}\label{ss-lin-global}
For an initial data in $H^s$ we get by Lemma \ref{controls} that the solution is globally in $H^s$, but with a growth of $\|u(t)\|_{H^s}$. To avoid this issue, we shall start with an initial data in a more restricted space. We recall the spaces defined in the Introduction by \eqref{Xtau} and \eqref{Ytau}. Let $0\leq\gamma<\frac 14$ throughout the rest of the paper. For a fixed $t_0$, we define a norm on functions depending only on space variable
$$\|f\|_{X_{t_0}^\gamma}=\frac{1}{t_0^\frac 14}\|f\|_{L^2}+\frac{t_0^\gamma}{\sqrt{t_0}}\||\xi|^{2\gamma}\hat{f}(\xi)\|_{L^\infty(\xi^2\leq 1)},$$
and a norm on functions depending on both time and space
$$\|g\|_{Y_{t_0}^\gamma}=\sup_{t\geq t_0}\,\left(\frac{1}{t_0^\frac 14}\|g(t)\|_{L^2}+\left(\frac{t_0}{t}\right)^{a^2}\frac{t_0^\gamma}{\sqrt{t_0}}\||\xi|^{2\gamma}\hat{g}(t,\xi)\|_{L^\infty(\xi^2\leq 1)}\right),$$
and $X_{t_0}^\gamma$ and $Y_{t_0}^\gamma$ are the corresponding spaces.

\begin{prop}\label{gl}
Let $t_0\geq 1$. Let $u(t_0)$ be a function in $X_{t_0}^\gamma$. Then there exists a unique global solution $u\in Y_{t_0}^\gamma$ of equation \eqref{lin} with $u(t_0)$ initial data at time $t_0$, and
$$\|u\|_{Y_{t_0}^\gamma}\leq  C(a)\, \|u(t_0)\|_{X_{t_0}^\gamma}.$$
More precisely, 
\begin{equation}\label{est}
\sup_{t\geq t_0}\,\frac{1}{t_0^\frac 14}\|u(t)\|_{L^2}\leq  C(a)\, \|u(t_0)\|_{X_{t_0}^\gamma}\,,\end{equation}
$$\sup_{t\geq t_0}\,\left(\frac{t_0}{t}\right)^{a^2}\frac{t_0^\gamma}{\sqrt{t_0}}\||\xi|^{2\gamma}\hat{u}(t,\xi)\|_{L^\infty(\xi^2\leq 1)}\leq  C\,\frac{t_0^\gamma}{\sqrt{t_0}}\||\xi|^{2\gamma}\hat{u}(t_0,\xi)\|_{L^\infty(\xi^2\leq 1)}.$$

\end{prop}

\begin{proof}
We first show the Proposition with $t_0=1$ and then for an arbitrary $t_0$.\\

We start with $u(1)\in X_1^\gamma$, which means that $u(1)\in L^2$ with $|\xi|^{2\gamma}\hat{u}(1,\xi)$ bounded in the region $\xi^2\leq 1$. We know already that a global solution $u(t)\in \mathcal{C}((1,\infty),L^2)$ exists, and we want to show that it belongs to $Y_1^\gamma$. By Lemma \ref{controls}, for all $M>0$,
\begin{equation}\label{linftyest}
\frac{1}{t^{a^2}}\||\xi|^{2\gamma}\hat{u}(t,\xi)\|_{L^\infty(\xi^2\leq M)}\leq 2\||\xi|^{2\gamma}\hat{u}(1,\xi)\|_{L^\infty(\xi^2\leq M)},
\end{equation}
so the second condition to be in $Y_1^\gamma$ is fulfilled by taking $M=1$. To control the $L^2$ norm we split it into two parts
$$\|u(t)\|_{L^2}=\|\hat{u}(t)\|_{L^2}=\|\hat{u}(t)\|_{L^2(\xi^2\leq 1)}+\|\hat{u}(t)\|_{L^2(1\leq \xi^2)}=I+J.$$
For both parts we use Lemma \ref{controls2}, with $\delta<\frac 14-\gamma$,
$$I\leq C(a)\||\xi|^{-2\delta}|\hat{u}(1,\xi)|\|_{L^2(\xi^2\leq 1)}\leq  C(a)\||\xi|^{-2(\gamma+\delta)}\|_{L^2(\xi^2\leq 1)}\,\||\xi|^{2\gamma}\hat{u}(1,\xi)\|_{L^\infty(\xi^2\leq 1)},$$
and
$$J\leq  C(a)\|\hat{u}(1,\xi)\|_{L^2(1\leq \xi^2)}\leq  C(a)\|\hat{u}(1)\|_{L^2}.$$
Therefore we have the $L^2$ norm of $u(t)$ bounded in time, 
$$\|u(t)\|_{L^2}\leq  C(a)\,\|u(1)\|_{L^2}+ C(a)\,\||\xi|^{2\gamma}\hat{u}(1,\xi)\|_{L^\infty(\xi^2\leq 1)}\leq  C(a)\|u(1)\|_{X_1^\gamma},$$
and so $u$ is in $Y_1^\gamma$.\\

Now we start with $u(t_0)\in X_{t_0}^\gamma$. We define $U(1)$ by
$$u(t_0,x)=U\left(1,\frac{x}{\sqrt{t_0}}\right).$$
We have 
$$\|u(t_0)\|_{L^2}=t_0^\frac 14\|U(1)\|_{L^2}$$
and
$$\frac{t_0^\gamma}{\sqrt{t_0}}\||\xi|^{2\gamma}\hat{u}(t_0,\xi)\|_{L^\infty(\xi^2\leq 1)}=
\frac{t_0^\gamma}{\sqrt{t_0}}\left\||\xi|^{2\gamma}\int e^{ix\xi}\,U\left(1,\frac{x}{\sqrt{t_0}}\right)dx\right\|_{L^\infty(\xi^2\leq 1)}$$
$$=t_0^\gamma\||\xi|^{2\gamma}\hat{U}(1,\xi\sqrt{t_0})\|_{L^\infty(\xi^2\leq 1)}=\||\xi|^{2\gamma}\hat{U}(1,\xi)\|_{L^\infty(\xi^2\leq t_0)}\geq \||\xi|^{2\gamma}\hat{U}(1,\xi)\|_{L^\infty(\xi^2\leq 1)}.$$
Hence 
$$\|U(1)\|_{X_1^\gamma}\leq\|u(t_0)\|_{X_{t_0}^\gamma},$$
and $U(1)$ is in $X_1^\gamma$. Therefore we can consider the global solution $U\in Y_1^\gamma$ of equation \eqref{lin} with initial data $U(1)$ at time 1. The function $u$ defined by 
$$u(t,x)=U\left(\frac{t}{t_0},\frac{x}{\sqrt{t_0}}\right)$$
is the solution of equation \eqref{lin} with initial data $u(t_0)$ at time $t_0$. We shall re-write the $L^2$ estimate and \eqref{linftyest} with $M=t_0$,
$$\sup_{t\geq 1}\,\|U(t)\|_{L^2}\leq C(a)\, \|U(1)\|_{X_1^\gamma},\quad \sup_{t\geq 1}\,\frac{1}{t^{a^2}}\||\xi|^{2\gamma}\hat{U}(t,\xi)\|_{L^\infty(\xi^2\leq t_0)}\leq 2\,\||\xi|^{2\gamma}\hat{U}(1,\xi)\|_{L^\infty(\xi^2\leq t_0)},$$
in terms of $u$. We have 
$$\sup_{t\geq 1}\,\|U(t)\|_{L^2}=\sup_{t\geq 1}\,\|u(t\,t_0,x\sqrt{t_0})\|_{L^2}=\sup_{t\geq 1}\,\frac{1}{t_0^\frac 14}\|u(t\,t_0)\|_{L^2}=\sup_{t\geq t_0}\,\frac{1}{t_0^\frac 14}\|u(t)\|_{L^2},$$
and since we have already shown that $\|U(1)\|_{X_1^\gamma}\leq\|u(t_0)\|_{X_{t_0}^\gamma}$, we get the first estimate of \eqref{est}. We have also already computed
$$\||\xi|^{2\gamma}\hat{U}(1,\xi)\|_{L^\infty(\xi^2\leq t_0)}=\frac{t_0^\gamma}{\sqrt{t_0}}\||\xi|^{2\gamma}\hat{u}(t_0,\xi)\|_{L^\infty(\xi^2\leq 1)},$$
and we get similarly
$$\sup_{t\geq 1}\,\frac{1}{t^{a^2}}\||\xi|^{2\gamma}\hat{U}(t,\xi)\|_{L^\infty(\xi^2\leq t_0)}=\sup_{t\geq 1}\,\frac{1}{t^{a^2}}\left\||\xi|^{2\gamma}\int e^{ix\xi}u(t\,t_0,x\sqrt{t_0})dx\right\|_{L^\infty(\xi^2\leq t_0)}$$
$$=\sup_{t\geq 1}\,\frac{1}{t^{a^2}}\frac{1}{\sqrt{t_0}}\left\||\xi|^{2\gamma}\hat{u}\left(t\,t_0,\frac{\xi}{\sqrt{t_0}}\right)\right\|_{L^\infty(\xi^2\leq t_0)}=\sup_{t\geq t_0}\,\left(\frac{t_0}{t}\right)^{a^2}\frac{1}{\sqrt{t_0}}\left\||\xi|^{2\gamma}\hat{u}\left(t,\frac{\xi}{\sqrt{t_0}}\right)\right\|_{L^\infty(\xi^2\leq t_0)}$$
$$=\sup_{t\geq t_0}\,\left(\frac{t_0}{t}\right)^{a^2}\frac{t_0^\gamma}{\sqrt{t_0}}\left\||\xi|^{2\gamma}\hat{u}(t,\xi)\right\|_{L^\infty(\xi^2\leq 1)},$$
so we get also the second estimate of \eqref{est} and the proof is complete.
\end{proof}

Since equation \eqref{lin} is linear, we can apply Proposition \ref{gl} for the higher order derivatives, and get the following statement.

\begin{corollary}\label{glHs} 
Let $s\in\mathbb N$ and $t_0\geq 1$. Let $u(t_0)$ be a function in $X_{t_0}^\gamma$ such that $\partial_x^k u(t_0)\in X_{t_0}^\gamma$ for all $0\leq k\leq s$. Then there exists a unique global solution $u\in Y_{t_0}^\gamma$ of equation \eqref{lin} with $u(t_0)$ initial data at time $t_0$, with $\partial_x^k u\in Y_{t_0}^\gamma$ for all $0\leq k\leq s$, and
$$\|\partial_x^k u\|_{Y_{t_0}^\gamma}\leq  C(a)\, \|\partial_x^k u(t_0)\|_{X_{t_0}^\gamma}.$$
More precisely, 
$$\sup_{t\geq t_0}\,\frac{1}{t_0^\frac 14}\|\partial_x^k u(t)\|_{L^2}\leq  C(a)\, \|\partial_x^k u(t_0)\|_{X_{t_0}^\gamma}\,,$$
$$\,\sup_{t\geq t_0}\,\left(\frac{t_0}{t}\right)^{a^2}\frac{t_0^\gamma}{\sqrt{t_0}}\||\xi|^{2\gamma}\widehat{\partial_x^k u}(t,\xi)\|_{L^\infty(\xi^2\leq 1)}\leq  C(a)\,\frac{t_0^\gamma}{\sqrt{t_0}}\||\xi|^{2\gamma}\widehat{\partial_x^k u}(t_0,\xi)\|_{L^\infty(\xi^2\leq 1)}.$$
\end{corollary}

\subsection{Asymptotic completeness}\label{ss-lin-asc}

\begin{prop}\label{linearasc}
Let $t_0\geq 1$ and let $u(t_0)$ be a function in $X_{t_0}^\gamma$. Then the unique global solution $u\in Y_{t_0}^\gamma$ of equation \eqref{lin} with $u(t_0)$ initial data at time $t_0$ scatters in $L^2$. 
More precisely, there exists $u_+\in L^2$ such that 
\begin{equation}\label{rate}
\|u(t)-e^{i(t-t_0)\partial_x^2}u_+\|_{L^2}\leq C(a,\delta)\,\frac{(1+\log t_0)t_0^{\frac 12-(\gamma+\delta)}}{t^{\frac 14-(\gamma+\delta)}}\|u(t_0)\|_{X_{t_0}^\gamma}\underset{t\tend\infty}{\longrightarrow} 0,
\end{equation}
for any $0<\delta<1/4-\gamma$.
\end{prop}

\begin{proof}
First we shall show that $e^{-i(t-t_0)\partial_x^2}\,u(t,x)$ has a limit in $L^2$ as $t$ goes to infinity. This is equivalent to 
$$\left\|e^{-it_2\partial_x^2}\,u(t_2,x)-e^{-it_1\partial_x^2}\,u(t_1,x)\right\|_{L^2}\underset{t_1,t_2\tend\infty}{\longrightarrow} 0,$$
and to
$$\left\|e^{it_2\xi^2}\,\hat{u}(t_2,\xi)-e^{it_2\xi^2}\,\hat{u}(t_1,\xi)\right\|_{L^2}=\left\|A_{t_1,t_2}(\xi)\right\|_{L^2}\underset{t_1,t_2\tend\infty}{\longrightarrow} 0.$$

For $1/t_0\leq \xi^2$, Lemma \ref{IBP} gives
$$\left\|A_{t_1,t_2}(\xi)\right\|_{L^2(1/t_0\leq \xi^2)}\leq C(a)\frac{t_0}{t_1}\|u(t_0)\|_{L^2}.$$

In the region $\xi^2\leq 1/t_2\leq 1/t_0$ we use Lemma \ref{controls2}
$$|A_{t_1,t_2}(\xi)|\leq a^2\int_{t_1}^{t_2}\,\frac{|\hat{u}(\tau,-\xi)|}{\tau}d\tau\leq C(a,\delta)\frac{|\hat{u}(t_0,-\xi)|+|\hat{u}(t_0,\xi)|}{(\xi^2\,t_0)^\delta}\,\log t_2,$$
so for $0<\delta<1/4-\gamma$,
$$\|A_{t_1,t_2}\|_{L^2(\xi^2\leq 1/t_2)}\leq C(a,\delta)\frac{\||\xi|^{2\gamma}\hat{u}(t_0,\xi)\|_{L^\infty(\xi^2\leq 1)}}{t_0^\delta} \left\|\frac{\log\xi^2}{\xi^{2(\gamma+\delta)}}\right\|_{L^2(\xi^2\leq 1/t_2)}$$
$$\leq C(a,\delta)\frac{1+\log t_2}{t_0^\delta\,t_2^{\frac 14-(\gamma+\delta)}}\||\xi|^{2\gamma}\hat{u}(t_0,\xi)\|_{L^\infty(\xi^2\leq 1)}.$$

In the region $1/t_2\leq \xi^2\leq 1/t_1\leq 1/t_0$, we split
$$A_{t_1,t_2}=A_{t_1,1/\xi^2}+A_{1/\xi^2,t_2}=I+J.$$
For $I$ we use again Lemma \ref{controls2}
$$|I|\leq a^2\int_{t_1}^{1/\xi^2}\,\frac{|\hat{u}(\tau,-\xi)|}{\tau}d\tau\leq C(a,\delta)\frac{|\hat{u}(t_0,\xi)|+|\hat{u}(t_0,-\xi)|}{(\xi^2\,t_0)^\delta}|\log\xi^2|,$$
and for $J$ we use Lemma \ref{IBP}
$$|J|\leq \frac{C(a,\delta)}{(\xi^2\,t_0)^\delta}\frac{|\hat{u}(t_0,\xi)|+|\hat{u}(t_0,-\xi)|}{\xi^2\frac{1}{\xi^2}}=C(a,\delta)\frac{|\hat{u}(t_0,\xi)|+|\hat{u}(t_0,-\xi)|}{(\xi^2\,t_0)^\delta}.$$
Then for $0<\delta<1/4-\gamma$
$$\|A_{t_1,t_2}\|_{L^2(1/t_2\leq \xi^2\leq 1/t_1)}\leq C(a,\delta)\frac{\||\xi|^{2\gamma}\hat{u}(t_0,\xi)\|_{L^\infty(\xi^2\leq 1/t_1)}}{t_0^\delta} \left\|\frac{\log\xi^2}{\xi^{2(\gamma+\delta)}}\right\|_{L^2(\xi^2\leq 1)}$$
$$\leq C(a,\delta)\frac{1+\log t_1}{t_0^\delta\,t_1^{\frac 14-(\gamma+\delta)}}\||\xi|^{2\gamma}\hat{u}(t_0,\xi)\|_{L^\infty(\xi^2\leq 1)}.$$

In the last region $1/t_1\leq \xi^2\leq 1/t_0$ we use Lemma \ref{IBP}
$$\|A_{t_1,t_2}\|_{L^2(1/t_1\leq\xi^2\leq 1/t_0)}\leq C(a,\delta)\frac{1}{t_1}\frac{\||\xi|^{2\gamma}\hat{u}(t_0,\xi)\|_{L^\infty(\xi^2\leq 1)}}{t_0^\delta}\left\|\frac{1}{\xi^{2+2(\gamma+\delta)}}\right\|_{L^2(1/t_1\leq \xi^2\leq 1)} $$
$$\leq C(a,\delta)\frac{\||\xi|^{2\gamma}\hat{u}(t_0,\xi)\|_{L^\infty(\xi^2\leq 1)}}{t_1\,t_0^\delta}\,t_1^{\frac 34+(\gamma+\delta)}=C(a,\delta)\frac{1}{t_0^\delta\,t_1^{\frac 14-(\gamma+\delta)}}\||\xi|^{2\gamma}\hat{u}(t_0,\xi)\|_{L^\infty(\xi^2\leq 1)}.$$

In conclusion, we have obtained
\begin{equation}\label{A}
\left\|A_{t_1,t_2}\right\|_{L^2}\leq C(a)\frac{t_0}{t_1}\|u(t_0)\|_{L^2}+C(a,\delta)\frac{1+\log t_1}{t_0^\delta\,t_1^{\frac 14-(\gamma+\delta)}}\||\xi|^{2\gamma}\hat{u}(t_0,\xi)\|_{L^\infty(\xi^2\leq 1)}
\end{equation}
$$\leq C(a,\delta)\left(t_0^\frac 14\frac{t_0}{t_1}+\frac{\sqrt{t_0}\,(1+\log t_1)}{t_0^{\gamma+\delta}\,t_1^{\frac 14-(\gamma+\delta)}}\right)\|u(t_0)\|_{X_{t_0}^\gamma}.$$
Therefore we have a limit $u_+\in L^2$ of $e^{-i(t-t_0)\partial_x^2}\,u(t,x)$ as $t$ goes to infinity. To get the decay rate \eqref{rate} we fix $t_1=t$ and $t_2=\infty$,
$$\left\|u_+-e^{-i(t-t_0)\partial_x^2}\,u(t,x)\right\|_{L^2}=\|A_{t,\infty}\|_{L^2}\leq C(a,\delta)\,t_0^{\frac 12-(\gamma+\delta)}\frac{1+\log t}{t^{\frac 14-(\gamma+\delta)}}\|u(t_0)\|_{X_{t_0}^\gamma},$$
for any $0<\delta<\frac 14-\gamma$, and since $t_0\geq 1$ the Proposition follows. 
\end{proof}

We have used in this proof the Lemmas \ref{controls}, \ref{controls2} and  \ref{IBP}, that are pointwise estimates in Fourier, so they apply to higher order derivatives. If $\partial_x^k u(t_0)\in X(t_0)^\gamma$, for $0\leq k\leq s$, we get then similar estimates as \eqref{A},
$$\left\|\partial_x^k\,A_{t_1,t_2}\right\|_{L^2}\leq C(a,\delta)\,\frac{(1+\log t_0)t_0^{\frac 12-(\gamma+\delta)}}{t_1^{\frac 14-(\gamma+\delta)}}\|\partial_x^k\,u(t_0)\|_{X_{t_0}^\gamma}.$$
Therefore we get a limit $u_+\in H^s$ of $e^{-i(t-t_0)\partial_x^2}\,u(t,x)$ as $t$ goes to infinity and
$$\left\|u_+-e^{-i(t-t_0)\partial_x^2}\,u(t,x)\right\|_{\dot{H}^k}=\|\partial_x^kA_{t,\infty}\|_{L^2}\leq C(a,\delta)\,\frac{(1+\log t_0)t_0^{\frac 12-(\gamma+\delta)}}{t^{\frac 14-(\gamma+\delta)}}\|\partial_x^ku(t_0)\|_{X_{t_0}^\gamma}.$$
Let us state this result.

\begin{corollary}
\label{linearascHs}
Let $s\in\mathbb N$ and $t_0\geq 1$. Let $u(t_0)$ be a function in $X_{t_0}^\gamma$ such that $\partial_x^k u(t_0)\in X_{t_0}^\gamma$ for all $0\leq k\leq s$. Then the unique global solution $u\in Y_{t_0}^\gamma$ of equation \eqref{lin} with $u(t_0)$ initial data at time $t_0$, with $\partial_x^k u\in Y_{t_0}^\gamma$ for all $0\leq k\leq s$, scatters in $H^s$. 
More precisely, there exists $u_+\in H^s$ such that
\begin{equation}\label{rateHs}
\|u(t)-e^{i(t-t_0)\partial_x^2}u_+\|_{\dot{H}^k}\leq C(a,\delta)\,\frac{(1+\log t_0)t_0^{\frac 12-(\gamma+\delta)}}{t^{\frac 14-(\gamma+\delta)}}\|\partial_x^ku(t_0)\|_{X_{t_0}^\gamma}\underset{t\tend\infty}{\longrightarrow} 0,
\end{equation}
for any $0<\delta<1/4-\gamma$.
\end{corollary}

\subsection{A-posteriori estimates}\label{ss-lin-aposteriori}
In this subsection we give some extra-estimates first on the asymptotic state $u_+$, and then on $u(t)$, solution of \eqref{lin} with initial condition $u(t_0)\in X_{t_0}$ and $t_0\geq 1$. 
By Proposition \ref{linearasc} we know already that $u_+\in L^2$ with
$$\|u_+\|_{L^2}\leq \|u(t)\|_{L^2}+C(a,\delta)\,\frac{(1+\log t_0)t_0^{\frac 12-(\gamma+\delta)}}{t^{\frac 14-(\gamma+\delta)}}\|u(t_0)\|_{X_{t_0}^\gamma},$$
for all $t\geq t_0\geq 1$, and by using \eqref{est} we obtain the bound
\begin{equation}\label{u+L2}
\|u_+\|_{L^2}\leq C(a)\,t_0^\frac 14\|u(t_0)\|_{X_{t_0}^\gamma}.
\end{equation}
Next we shall derive a control of the asymptotic state $u_+$ in the spirit of the one in Lemma \ref{controls2} on the solution $u(t)$.

\begin{lemma}\label{u+}
Let $0<\delta$.
The function $u_+$ satisfies for all $\xi\neq 0$ the estimate
\begin{equation}\label{u+mod}
|\hat{u}_+(\xi)|\leq\left(C(a)+C(a,\delta)\frac{1+|\log|\xi||}{(\xi^2\,t_0)^\delta}\right)\left(|\hat{u}(t_0,\xi)|+|\hat{u}(t_0,-\xi)|\right).
\end{equation}
\end{lemma}

\begin{proof}
We have in $L^2$ and pointwise in Fourier variables,
\begin{equation}\label{Duhamelu+}
u_+(x)=u(t_0,x)+ia^2\int_{t_0}^\infty e^{-i\tau\partial_x^2}\frac{\overline{u(\tau,x)}}{\tau^{1\pm 2ia^2}}d\tau,
\end{equation}
so
$$\hat{u_+}(\xi)=\hat{u}(t_0,\xi)+e^{it\xi^2}A_{t_0,\infty}(\xi),$$
and
$$|\hat{u_+}(\xi)|\leq|\hat{u}(t_0,\xi)|+|A_{t_0,\infty}(\xi)|.$$

For the region $1/t_0\leq\xi^2$ the conclusion follows immediately from  Lemma \ref{IBP}. 

For the region $\xi^2\leq 1/t_0$ we have obtained in the proof of Proposition \ref{linearasc} that 
$$|A_{t_0,\infty}(\xi)|\leq C(a,\delta)\frac{|\hat{u}(t_0,-\xi)|+|\hat{u}(t_0,\xi)|}{(\xi^2\,t_0)^\delta}|\log\xi^2|+C(a,\delta)\frac{|\hat{u}(t_0,\xi)|+|\hat{u}(t_0,-\xi)|}{(\xi^2\,t_0)^\delta},$$
and the Lemma follows.

\end{proof}

In particular, for all $\xi^2\leq 1/t_0$ we have
$$\frac{|\xi|^{2(\gamma+\delta)}}{1+|\log|\xi||}|\hat{u_+}(\xi)|\leq C(a,\delta)\, t_0^{\frac 12-(\gamma+\delta)}\|u(t_0)\|_{X_{t_0}^\gamma},$$
for any $0<\delta$. So, if $t_0=1$, we get for all $\xi^2\leq 1$ and for any $0<\delta$,
\begin{equation}\label{fourierlin}
|\xi|^{2(\gamma+\delta)}|\hat{u_+}(\xi)|\leq C(a,\delta)\, \|u(1)\|_{X_{1}^\gamma}.
\end{equation}

We end this section with a regularity property of the solutions of \eqref{lin}.

\begin{prop}\label{str}
Under the assumptions of Proposition \ref{linearasc}, the solution $u(t)$ belongs to $L^4((t_0,\infty),L^\infty)$ with the bound
$$\|u\|_{L^4((t_0,\infty),L^\infty)}\leq C(a)\,t_0^{\frac 14}\,(1+\log^2 t_0)\,\|u(t_0)\|_{X_{t_0}^\gamma},$$
and so does also $u(t)-e^{i(t-t_0)\partial_x^2}u_+$.
\end{prop}

\begin{proof}
We use the Duhamel formulae
$$u(t)=e^{i(t-t_0)\partial_x^2}u(t_0)+ia^2\int_{t_0}^t e^{i(t-\tau)\partial_x^2}\frac{\overline{u(\tau)}}{\tau^{1\pm 2ia^2}}d\tau$$
$$=e^{i(t-t_0)\partial_x^2}u(t_0)+ia^2\int_{t_0}^t e^{i(t-\tau)\partial_x^2}\frac{\overline{u(\tau)-e^{i(\tau-t_0)\partial_x^2}u_+}}{\tau^{1\pm 2ia^2}}d\tau+ia^2\int_{t_0}^t e^{i(t-2\tau)\partial_x^2}\frac{e^{it_0\partial_x^2}\overline{u_+}}{\tau^{1\pm 2ia^2}}d\tau.$$
Since $(4,\infty)$ is a Strichartz 1-d admissible couple, we can upper-bound the $L^4((t_0,\infty),L^\infty)$ norm of the first and of the second term by
$$M=C\|u(t_0)\|_{L^2}+a^2\int_{t_0}^\infty\frac{\|u(\tau)-e^{i(\tau-t_0)\partial_x^2}u_+\|_{L^2}}{\tau}d\tau,$$
and by using the rate of decay of Proposition \ref{linearasc}, for some $0<\delta<\frac 14-\gamma$,
$$M\leq C\|u(t_0)\|_{L^2}+C(a)\,(1+\log t_0)t_0^{\frac 12-(\gamma+\delta)}\|u(t_0)\|_{X_{t_0}}\int_{t_0}^\infty\frac{d\tau}{\tau^{\frac 54-(\gamma+\delta)}}$$
$$\leq C(a)\,(1+\log t_0)t_0^{\frac 14}\,\|u(t_0)\|_{X_{t_0}^\gamma}.$$
Therefore we only need to estimate in $L^4((t_0,\infty),L^\infty)$ the last term. Let $\theta(x)$ be a cut-off function with $\theta(x)=0$ for $|x|<\frac 12$ and $\theta(x)=1$ for $|x|>1$. We decompose as usual the domain of the Fourier variable into three regions, $\xi^2\ls 1/t$, $1/t\leq\xi^2\leq 1/t_0$ and $1/t_0\leq\xi^2$,
$$\int_{t_0}^t e^{i(t-2\tau)\partial_x^2}\frac{e^{it_0\partial_x^2}\overline{u_+}}{\tau^{1\pm 2ia^2}}d\tau=\int e^{ix\xi}\,e^{-it\xi^2}\,e^{-it_0\xi^2}\,\overline{\hat{u_+}(-\xi)}  \int_{t_0}^{t}\frac{e^{i2\tau\xi^2}}{\tau^{1\pm 2ia^2}}d\tau d\xi$$
$$=\int (1-\theta)(t\xi^2)+\int \theta(t\xi^2)\,(1-\theta)(t_0\xi^2)+\int \theta(t\xi^2)\,\theta(t_0\xi^2)=I+J+K.$$

For $I$ we integrate directly in $\tau$, 
$$|I(t)|\leq \int_{\xi^2\leq 1/t} |\hat{u_+}(-\xi)|\log t \,d\xi$$
and we apply Lemma \ref{u+}, for some $0<\delta<\frac 14-\gamma$,
$$|I(t)|\leq C(a)\frac{\log t}{t_0^{\delta}}\int_{\xi^2\leq 1/t}\frac{1+|\log|\xi||}{\xi^{2\delta}}(|\hat{u}(t_0,\xi)|+|\hat{u}(t_0,-\xi)|)\,d\xi
\leq C(a)\frac{\||\xi|^{2\gamma}\hat{u}(t_0,\xi)\|_{L^\infty(\xi^2\leq 1)}}{t_0^{\delta}}\,\frac{\log^2 t}{t^{\frac 12-(\gamma+\delta)}}.$$
Then
$$\|I\|_{L^4((t_0,\infty),L^\infty)}\leq C(a)\,\frac{\||\xi|^{2\gamma}\hat{u}(t_0,\xi)\|_{L^\infty(\xi^2\leq 1)}}{t_0^{\delta}}\,\frac{1+\log^2 t_0}{t_0^{\frac 14-(\gamma+\delta)}}\leq C(a)\,t_0^{\frac 14}\,(1+\log^2 t_0)\,\|u(t_0)\|_{X_{t_0}^\gamma}.$$

To treat $J$ we first split the integral in $\tau$ into two parts
$$J=\int e^{ix\xi}\,e^{-it\xi^2}\,\theta(t\xi^2)\,(1-\theta)(t_0\xi^2)\,e^{-it_0\xi^2}\,\overline{\hat{u_+}(-\xi)}  \int_{t_0}^{1/\xi^2}\frac{e^{i2\tau\xi^2}}{\tau}d\tau d\xi$$
$$+\int e^{ix\xi}\,e^{-it\xi^2}\,\theta(t\xi^2)\,(1-\theta)(t_0\xi^2)\,e^{-it_0\xi^2}\,\overline{\hat{u_+}(-\xi)}  \int_{1/\xi^2}^t\frac{e^{i2\tau\xi^2}}{\tau}d\tau d\xi=J_1+J_2.$$
We need the following lemma.
\begin{lemma}
Define $U_tf$ as $\widehat{U_tf}(\xi)=\phi (\sqrt{|t|}\xi)\,e^{-it\xi^2}\hat{f}(\xi)$, with $\|\phi\|_{L^\infty}+\|\phi'\|_{L^1}\leq C$. Then
$$\|U_tf\|_{L^4_tL^\infty_x}\leq C\|f\|_{L^2}.$$
\end{lemma}
\begin{proof}
The lemma follows from the usual $TT^*$ argument and the elementary inequality
$$\int e^{-it\xi^2+ix\,\xi}\phi(\sqrt{|t|}\xi)\leq \frac{C}{\sqrt{|t|}}\left(\|\phi\|_{L^\infty}+\|\phi'\|_{L^1}\right).$$
\end{proof}
Therefore we get the following estimate for $J_1$
$$\|J_1\|_{L^4((t_0,\infty),L^\infty)}\leq C\,\left\|(1-\theta)(t_0\,\xi^2)\,\overline{\hat{u_+}(-\xi)}  \int_{t_0}^{1/\xi^2}\frac{e^{i2\tau\xi^2}}{\tau}d\tau \right\|_{L^2}\leq C\,\left\|\hat{u_+}(-\xi) \log |\xi|\right\|_{L^2(\xi^2\leq 1/t_0)}.$$
Now we use Lemma \ref{u+} and get
$$\|J_1\|_{L^4((t_0,\infty),L^\infty)}\leq C(a)\,\frac{\||\xi|^{2\gamma}\hat{u}(t_0,\xi)\|_{L^\infty(\xi^2\leq 1)}}{t_0^\delta}\,\left\|\frac{1+\log^2|\xi|}{\xi^{2(\gamma+\delta)}}\right\|_{L^2(\xi^2\leq 1/t_0)}$$
$$\leq C(a)\,\frac{1+\log^2 t_0}{t_0^{{\frac 14}-\gamma}}\,\||\xi|^{2\gamma}\hat{u}(t_0,\xi)\|_{L^\infty(\xi^2\leq 1)}\leq C(a)\,t_0^{\frac 14}\,(1+\log^2 t_0)\,\|u(t_0)\|_{X_{t_0}^\gamma}.$$
For $J_2$ we perform first the integration by parts
$$\int_{1/\xi^2}^t\frac{e^{i2\tau\xi^2}}{\tau}d\tau=
\left.\frac {e^{i2\tau\xi^2}}{2i\,\xi^2\tau}\right|_{1/\xi^2}^t+\int _{1/\xi^2}^{t}\frac{e^{i2\tau\xi^2}}{2i\,\xi^2\tau^2}d\tau=\frac {e^{i2t\xi^2}}{2i\,\xi^2t}-\frac {e^{i2}}{2i}+\int _{1}^{t\xi^2}\frac{e^{i2\tau}}{2i\tau^2}d\tau$$
$$=\frac {e^{i2t\xi^2}}{2i\,\xi^2t}-\int _{t\xi^2}^{\infty}\frac{e^{i2\tau}}{2i\tau^2}d\tau-\frac {e^{i2}}{2i}+\int _{1}^{\infty}\frac{e^{i2\tau}}{2i\tau^2}d\tau.$$
Therefore
$$|J_2(t)|\leq \frac{C}{t}\int_{1/2t\leq \xi^2\leq 1/t_0}\frac{|\hat{u_+}(-\xi)|}{\xi^2}\,  d\xi+C\left|\int e^{ix\xi}\,e^{-it\xi^2}\,\theta(t\,\xi^2)\,(1-\theta)(t_0\,\xi^2)\,e^{-it_0\xi^2}\,\overline{\hat{u_+}(-\xi)} d\xi\right|.$$
For the first term we use again Lemma \ref{u+}, and get
$$\frac{C}{t}\int_{1/2t \leq \xi^2\leq 1/t_0}\frac{|\hat{u_+}(-\xi)|}{\xi^2}\,  d\xi\leq C(a)\,\frac{\||\xi|^{2\gamma}\hat{u}(t_0,\xi)\|_{L^\infty(\xi^2\leq 1)}}{t\,t_0^\delta}\,\left\|\frac{1+|\log|\xi||}{\xi^{2+2(\gamma+\delta)}}\right\|_{L^1(1/2t \leq\xi^2\leq 1/t_0)}$$
$$\leq C(a) \frac{1+\log t}{t^{\frac 12-(\gamma+\delta)}}\,\frac{\||\xi|^{2\gamma}\hat{u}(t_0,\xi)\|_{L^\infty(\xi^2\leq 1)}}{t_0^\delta}\,.$$
The second term of $J_2$ is similar to a linear evolution as $J_1$. We obtain
$$\|J_2\|_{L^4((t_0,\infty),L^\infty)}\leq C(a)\,\frac{1+\log^2 t_0}{t_0^{\frac 14-\gamma}}\,\||\xi|^{2\gamma}\hat{u}(t_0,\xi)\|_{L^\infty(\xi^2\leq 1)},$$
so 
$$\|J\|_{L^4((t_0,\infty),L^\infty)}\leq C(a)\,t_0^{\frac 14}\,(1+\log^2 t_0)\,\|u(t_0)\|_{X_{t_0}^\gamma}.$$

For $K$ we use again the integration by parts
$$\int_{t_0}^t\frac{e^{i2\tau\xi^2}}{\tau}d\tau=\frac {e^{i2t\xi^2}}{2i\,\xi^2t}-\int _{t\xi^2}^{\infty}\frac{e^{i2\tau}}{2i\tau^2}d\tau-\frac {e^{i2t_0\xi^2}}{2i\,\xi^2t_0}+\int _{t_0}^{\infty}\frac{e^{i2\tau\xi^2}}{2i\,\xi^2\tau^2}d\tau,$$
hence
$$|K(t)|\leq \frac{C}{t}\int_{1/2t_0\leq \xi^2}\frac{|\hat{u_+}(-\xi)|}{\xi^2}\,  d\xi$$
$$+\left|\int e^{ix\xi}\,e^{-it\xi^2}\,\theta(t\,\xi^2)\,\,\theta(t_0\,\xi^2)\,e^{-it_0\xi^2}\,\overline{\hat{u_+}(-\xi)} \left(-\frac {e^{i2t_0\xi^2}}{2i\,\xi^2t_0}+\int _{t_0}^{\infty}\frac{e^{i2\tau\xi^2}}{2i\,\xi^2\tau^2}d\tau\right) d\xi\right|.$$
By Cauchy-Schwarz's inequality, the first term is upper-bounded by $C\frac{t_0^\frac 34}{t}\|u_+\|_{L^2}$. By \eqref{u+L2} this in turn is smaller than $C(a)\frac{t_0}{t}\,\|u(t_0)\|_{X_0^\gamma}$.
We get again, as for $J_2$,
$$\|K\|_{L^4((t_0,\infty),L^\infty)}\leq C(a)\,t_0^{\frac 14}\,(1+\log^2 t_0)\,\|u(t_0)\|_{X_{t_0}^\gamma}.$$

Summarizing, we have obtained the desired estimate
$$\|u\|_{L^4((t_0,\infty),L^\infty)}\leq C(a)\,t_0^{\frac 14}\,(1+\log^2 t_0)\,\|u(t_0)\|_{X_{t_0}^\gamma}.$$
The Strichartz inequalities for a free evolution together with \eqref{u+L2} give
$$\|e^{i(t-t_0)\partial_x^2}u_+\|_{L^4((t_0,\infty),L^\infty)}\leq C\|u_+\|_{L^2}\leq C(a)\,t_0^{{\frac 14}}\|u(t_0)\|_{X_{t_0}^\gamma},$$
so we also have 
$$\|u(t)-e^{i(t-t_0)\partial_x^2}u_+\|_{L^4((t_0,\infty),L^\infty)}\leq C(a)\,t_0^{{\frac 14}}(1+\log^2 t_0)\,\|u(t_0)\|_{X_{t_0}^\gamma}.$$
\end{proof}

Lemma \ref{u+} is a pointwise estimate for Fourier transforms, so it  fits for higher order derivatives. 
Again by linearity we have the results of Proposition \ref{str} at higher Sobolev order, if $\partial_x^k u(t_0)\in X(t_0)$ :
$\partial_x^k u(t)$ belongs to $L^4((t_0,\infty),L^\infty)$ with the bound
$$\|\partial_x^k u\|_{L^4((t_0,\infty),L^\infty)}\leq C(a)\,t_0^{{\frac 14}}(1+\log^2 t_0)\,\|\partial_x^k u(t_0)\|_{X_{t_0}^\gamma}.$$

\section{Scattering for the nonlinear equation}\label{s:nonlinear}
In this section we prove Theorem \ref{theorem1}. By using the results on the linear equation \eqref{lin} obtained in the previous section, we first infer in \S\ref{ss-nonlin-global} a global existence result for the nonlinear equation \eqref{nonlin}. Then we prove in \S\ref{ss-nonlin-asc} asymptotic completeness for these solutions. In the last subsection \S\ref{ss-nonlin-aposteriori} we give new information about the regularity of the asymptotic state, which completes the proof of Theorem \ref{theorem1}.\\

We start by writing the nonlinear solutions of \eqref{nonlin} in terms of solutions of the linear equation \eqref{lin}. Let us notice that the estimates obtained in the previous section are independent of the sign in \eqref{lin}, so in the sequel we shall consider only one of the signs - the other sign case can be treated the same. We denote by $S(t,t_0)f$ the solution of \eqref{lin} with sign plus
$$iu_t+u_{xx}+\frac{a^2}{t^{1+ 2ia^2}}\overline{u}=0,$$
with initial data $f$ at time $t_0\geq 1$. With this notation, for $t_0\leq t$ we have the estimates \eqref{est} of Proposition \ref{gl},
\begin{equation}\label{estL2}
\|S(t,t_0)f\|_{L^2}\leq C(a)\, t_0^\frac 14\,\|f\|_{X_{t_0}^\gamma},
\end{equation}
and
\begin{equation}\label{estmod}
\||\xi|^{2\gamma}\widehat{S(t,t_0)f}(\xi)\|_{L^\infty(\xi^2\leq 1)}\leq C\,\left(\frac{t}{t_0}\right)^{a^2}\||\xi|^{2\gamma}\hat{f}(\xi)\|_{L^\infty(\xi^2\leq 1)},
\end{equation}
and the one of Proposition \ref{str},
\begin{equation}\label{eststr}
\|S(\cdot,t_0)f\|_{L^4((t_0,\infty),L^\infty)}\leq C(a)\,t_0^{{\frac 14}}(1+\log^2 t_0)\,\|f\|_{X_{t_0}^\gamma},
\end{equation}
as well as all their equivalents at higher order derivatives, if $\partial_x^k f\in X_{t_0}^\gamma$.

Now the solution of 
$$u_t=i\left(u_{xx}+\frac{a^2}{t^{1+ 2ia^2}}\overline{u}+ \frac{F}{t}\right)$$
with $u(1)$ initial data at time $t=1$ writes 
\begin{equation}\label{nonlinformula}
u(t,x)=S(t,1)\,u(1)+\int_{1}^t S(t,\tau) \frac{iF\,(\tau)}{\tau} d\tau.
\end{equation}
It is enough to verify this formula for $u(1)=0$,  
$$\partial_t u=\partial_t\int_{1}^t S(t,\tau) \frac{iF\,(\tau)}{\tau} d\tau=i\frac{F}{t}+\int_{1}^ti\left(\partial_{xx}S(t,\tau) \frac{iF\,(\tau)}{\tau}+ \frac{a^2}{t^{1+ 2ia^2}}\overline{S(t,\tau) \frac{iF\,(\tau)}{\tau}}\right)d\tau$$
$$=i\left(\frac{F}{t}+u_{xx}+\frac{a^2}{t^{1+ 2ia^2}} \overline{u}\right).$$
In our case of \eqref{nonlin}, $F$ is composed of cubic and quadratic powers of $u$.

\subsection{Global existence}\label{ss-nonlin-global}

Let us recall again the definitions of the norms of $X_{1}^\gamma$ and $Y_{1}^\gamma$, for $0\leq\gamma<\frac14$,
$$\|f\|_{X_1^\gamma}=\|f\|_{L^2}+\||\xi|^{2\gamma}\hat{f}(\xi)\|_{L^\infty(\xi^2\leq 1)},$$
$$\|g\|_{Y_1^\gamma}=\sup_{t\geq 1}\,\left(\|g(t)\|_{L^2}+\frac{1}{t^{a^2}}\||\xi|^{2\gamma}\hat{g}(t,\xi)\|_{L^\infty(\xi^2\leq 1)}\right).$$

We have the following global existence result on the nonlinear equation \eqref{nonlin}.
\begin{prop}\label{global}
Let $u(1)$ be a function in $X_1^\gamma$ small with respect to $a$. Then there exists a unique global solution $u\in Z^\gamma=Y_1^\gamma\cap L^4((1,\infty),L^\infty)$ of equation \eqref{nonlin} with $u(1)$ initial data at time $t=1$, and
$$\|u\|_{Z^\gamma}\leq C(a)\, \|u(1)\|_{X_1^\gamma}.$$
\end{prop}

\begin{proof}
In view of \eqref{nonlinformula} we shall prove the propositions by doing a fixed point argument in $Z^\gamma$ for the operator 
$$\Phi(u)(t)=S(t,1)\,u(1)+\int_{1}^t S(t,\tau) \frac{iF(u(\tau))}{\tau} d\tau.$$
The estimates \eqref{estL2}, \eqref{estmod}, and \eqref{eststr} ensure us that
$$\|S(t,1)\,u(1)\|_{Z^\gamma}\leq C(a)\, \|u(1)\|_{X_1^\gamma}.$$

We start with a property that we shall use frequently in the following.
\begin{lemma}\label{Duh23}
Let $u\in Z^\gamma$ and $\alpha<\frac 12-\gamma$. Then for $1\leq t_1\leq t_2$
$$\int_{t_1}^{t_2}\tau^\alpha\,\frac{\|F(u(\tau))\|_{X_\tau^\gamma}}{\tau} d\tau\leq C\frac{\Sigma_{j\in\{1,2\}}\left(a\|u\|^2_{L^{p_j}((t_1,t_2),L^{q_j})}+\|u\|^3_{L^{p_j}((t_1,t_2),L^{q_j})}\right)}{t_1^{\frac 12-\alpha-\gamma}}\leq C\,\frac{a\|u\|_{Z^\gamma}^2+\|u\|_{Z^\gamma}^3}{t_1^{\frac 12-\alpha-\gamma}},$$
where $(p_1,q_1)=(\infty,2)$ and $(p_2,q_2)=(4,\infty)$.\end{lemma}

\begin{proof}
By definition \eqref{Xtau} of $X_\tau^\gamma$, and since $|\hat{f}(\xi)|\leq \|f\|_{L^1}$, we get
$$\int_{t_1}^{t_2}\tau^\alpha\,\frac{\|F(u(\tau))\|_{X_\tau^\gamma}}{\tau} d\tau=
\int_{t_1}^{t_2}\left(\frac{1}{\tau^\frac 14}\|F(u(\tau))\|_{L^2}+\frac{\tau^\gamma}{\sqrt{\tau}}\||\xi|^{2\gamma}\widehat{F(u(\tau))}\|_{L^\infty(\xi^2\leq 1)}\right)\frac{d\tau}{\tau^{1-\alpha}}$$
$$\leq c\, a\int_{t_1}^{t_2} \|u(\tau)\|_{L^4}^2\frac{d\tau}{\tau^{\frac 54-\alpha}}+c\,a\int_{t_1}^{t_2}\|u(\tau)\|_{L^2}^2\frac{d\tau}{\tau^{\frac32-\alpha-\gamma}}+c\,\int_{t_1}^{t_2}\|u(\tau)\|_{L^6}^3\frac{d\tau}{\tau^{\frac 54-\alpha}}+c\,\int_{t_1}^{t_2}\|u(\tau)\|_{L^3}^3\frac{d\tau}{\tau^{\frac32-\alpha-\gamma}}.$$
We apply H\"older's inequality $L^4-L^\frac 43$ in the first and the last integral, and Cauchy-Schwarz's inequality for the third one.
$$\int_{t_1}^{t_2}\tau^\alpha\,\frac{\|F(u(\tau))\|_{X_\tau^\gamma}}{\tau} d\tau\leq c\,a\,\|u\|_{L^8((t_1,t_2),L^4)}^2\left\|\frac{1}{\tau^{\frac 54-\alpha}}\right\|_{L^\frac43(t_1,t_2)}+c\,a\,\|u\|_{L^\infty((1,\infty),L^2)}^2\int_{t_1}^{t_2}\frac{d\tau}{\tau^{\frac32-\alpha-\gamma}}$$
$$+c\,\|u\|_{L^6((t_1,t_2),L^6)}^3\left\|\frac{1}{\tau^{\frac 54-\alpha}}\right\|_{L^2(t_1,t_2)}+c\,\|u\|_{L^{12}((t_1,t_2),L^3)}^3\left\|\frac{1}{\tau^{\frac 32-\alpha-\gamma}}\right\|_{L^\frac43(t_1,t_2)}.$$
The spaces $L^8L^4$, $L^6L^6$ and $L^{12}L^3$ are interpolation spaces between $L^\infty L^2$ and $L^4L^\infty$, therefore the Lemma follows.

\end{proof}

Let $u\in Z^\gamma$. 
The $L^4L^\infty$ norm of the integral in $\Phi(u)$ can be bounded by
$$\left\|a\int_{1}^t S(t,\tau) \frac{iF(u(\tau))}{\tau} d\tau\right\|_{L^4((1,\infty),L^\infty)}\leq a\int_{1}^\infty\left\| S(t,\tau) \frac{iF(u(\tau))}{\tau} \right\|_{L^4((\tau,\infty),L^\infty)}d\tau.$$
By using \eqref{eststr},
$$\left\|a\int_{1}^t S(t,\tau) \frac{iF(u(\tau))}{\tau} d\tau\right\|_{L^4((1,\infty),L^\infty)}\leq C(a)\,\int_{1}^\infty\tau^{\frac14}(1+\log^2 \tau)\,\frac{\|F(u(\tau))\|_{X_\tau^\gamma}}{\tau} d\tau,$$
so Lemma \ref{Duh23} with $\alpha=\frac 14^+$ gives us
$$\left\|a\int_{1}^t S(t,\tau) \frac{iF(u(\tau))}{\tau} d\tau\right\|_{L^4((1,\infty),L^\infty)}\leq C(a)\,(\|u\|_{Z^\gamma}^2+\|u\|_{Z^\gamma}^3).$$

Next we upper-bound the $L^\infty L^2$ norm 
$$\left\|a\int_{1}^t S(t,\tau) \frac{iF(u(\tau))}{\tau} d\tau\right\|_{L^2}\leq a\int_{1}^\infty\left\| S(t,\tau) \frac{iF(u(\tau))}{\tau} \right\|_{L^2}d\tau,$$
and by using \eqref{estL2},
$$\left\|a\int_{1}^t S(t,\tau) \frac{iF(u(\tau))}{\tau} d\tau\right\|_{L^2}\leq C(a)\,\int_{1}^t \tau^\frac 14\frac{\|F(u(\tau))\|_{X_\tau^\gamma}}{\tau} d\tau.$$
Again, Lemma \ref{Duh23} with $\alpha=\frac 14$ gives us
$$\left\|a\int_{1}^t S(t,\tau) \frac{iF(u(\tau))}{\tau} d\tau\right\|_{L^2}\leq C(a)\,(\|u\|_{Z^\gamma}^2+c\,\|u\|_{Z^\gamma}^3).$$

Finally, we compute (the contribution of the other quadratic term $|u|^2$ can be treated the same)
$$\frac{1}{t^{a^2}}\left\||\xi|^{2\gamma}\mathcal{F}\left(a\int_{1}^t S(t,\tau) \frac{iu^2\,(\tau)}{\tau^{1-ia^2}} d\tau\right)\right\|_{L^\infty(\xi^2\leq 1)}\leq \frac{a}{t^{a^2}}\int_{1}^t \left\||\xi|^{2\gamma}\mathcal{F}\left(S(t,\tau) \frac{iu^2\,(\tau)}{\tau^{1-ia^2}} \right)\right\|_{L^\infty(\xi^2\leq 1)}d\tau,$$
and by \eqref{estmod}
$$\frac{1}{t^{a^2}}\left\||\xi|^{2\gamma}\mathcal{F}\left(a\int_{1}^t S(t,\tau) \frac{iu^2\,(\tau)}{\tau^{1-ia^2}} d\tau\right)\right\|_{L^\infty(\xi^2\leq 1)}\leq \frac{Ca}{t^{a^2}}\int_{1}^t \left(\frac{t}{\tau}\right)^{a^2}\||\xi|^{2\gamma}\widehat{u^2}(\tau,\xi)\|_{L^\infty(\xi^2\leq 1)}\frac{d\tau}{\tau},$$
$$\leq Ca\,\|u\|_{L^\infty((1,\infty),L^2)}^2\int_1^\infty\frac{d\tau}{\tau^{1+a^2}}\leq \frac{C}{a}\,\|u\|_{Z^\gamma}^2.$$

Also, by \eqref{estmod} and by H\"older's inequality,
$$\frac{1}{t^{a^2}}\left\||\xi|^{2\gamma}\mathcal{F}\left(\int_{1}^t S(t,\tau) \frac{i|u|^2u\,(\tau)}{\tau} d\tau\right)\right\|_{L^\infty(\xi^2\leq 1)}\leq \frac{C}{t^{a^2}}\int_{1}^t \left\||\xi|^{2\gamma}\mathcal{F}\left(S(t,\tau) \frac{i|u|^2u\,(\tau)}{\tau} \right)\right\|_{L^\infty(\xi^2\leq 1)}d\tau$$
$$\leq \frac{C}{t^{a^2}}\int_{1}^t \left(\frac{t}{\tau}\right)^{a^2}\||\xi|^{2\gamma}\widehat{|u|^2u}(\tau,\xi)\|_{L^\infty(\xi^2\leq 1)}\frac{d\tau}{\tau}\leq C\,\int_1^\infty\|u(\tau)\|_{L^3}^3\frac{d\tau}{\tau^{1+a^2}}\leq C\,\|u\|_{L^{12}((1,\infty),L^3)}^3.$$
So we have shown that the contribution of the quadratic and cubic term is in $Z^\gamma$, 
$$\left\|\int_{1}^t S(t,\tau) \frac{iF(u(\tau))}{\tau} d\tau\right\|_{Z^\gamma}\leq C(a)\,(\|u\|_{Z^\gamma}^2+\|u\|_{Z^\gamma}^3).$$

Summarizing, we have
$$\|\Phi(u)\|_{Z^\gamma}\leq C(a)\,(\|u(1)\|_{X_1^\gamma}+\|u\|_{Z^\gamma}^2+\|u\|_{Z^\gamma}^3),$$
so for $u(1)\in X_1^\gamma$ small with respect to $a$, by the fixed point argument we get a global solution of \eqref{nonlin} $u\in Z^\gamma$ with norm bounded by
$$\|u\|_{Z^\gamma}\leq C(a)\,\|u(1)\|_{X_1^\gamma}.$$
\end{proof}

We state now the result in Sobolev spaces. This is a direct corollary of Proposition \ref{global}, by using the Leibniz rule and the fact that estimating the Fourier norm in $Z^\gamma$ on derivative terms creates powers of $\xi$ which are bounded by $1$.

\begin{corollary}\label{globalHs}
Let $s\in\mathbb N$. Let $\partial_x^ku(1)$ be a function in $X_1^\gamma$ small with respect to $a$, for all $0\leq k\leq s$. Then there exists a unique global solution $u\in Z^\gamma=Y_1^\gamma\cap L^4((1,\infty),L^\infty)$, with $\partial_x^k u\in  Z^\gamma$, of equation \eqref{nonlin} with $u(1)$ initial data at time $t=1$, and
$$\Sigma_{0\leq k\leq s}\,\|\partial_x^k u\|_{Z^\gamma}\leq C(a)\, \Sigma_{0\leq k\leq s}\,\|\partial_x^k u(1)\|_{X_1^\gamma}.$$
\end{corollary}

\subsection{Asymptotic completeness}\label{ss-nonlin-asc}
Now we prove the second part of Theorem \ref{theorem1}, namely the asymptotic completeness of the global solutions obtained by Proposition \ref{global}.
\begin{prop}\label{nonlinearasc}
Let $u(1)$ be a function in $X_1^\gamma$ small with respect to $a$. Then the unique global solution $u\in  Z^\gamma=Y_1^\gamma\cap L^4((1,\infty),L^\infty)$ of equation \eqref{nonlin} with $u(1)$ initial data at time $t=1$ scatters in $L^2$. 
More precisely, there exists $f_+\in L^2$ for which
\begin{equation}\label{ratenonlin}
\|u(t)-e^{i(t-1)\partial_x^2}f_+\|_{L^2}\leq \frac{C(a,\delta)}{t^{\frac14-(\gamma+\delta)}}\,\|u(1)\|_{X_1^\gamma}\underset{t\tend\infty}{\longrightarrow} 0,
\end{equation}
for any $0<\delta<1/4-\gamma$. 
\end{prop}

\begin{proof}
The nonlinear solution writes
$$u(t)=S(t,1)\,u(1)+\int_{1}^t S(t,\tau) \frac{iF(u(\tau))}{\tau} d\tau.$$
The scattering result of Proposition \ref{linearasc} guarantees the existence of $u_+\in L^2$ such that
$$\|S(t,1)\,u(1)-e^{i(t-1)\partial_x^2}u_+\|_{L^2}\leq \frac {C(a,
\tilde\delta)}{t^{\frac14-(\gamma+\tilde\delta)}}\,\|u(1)\|_{X_1^\gamma},$$
for some $\tilde\delta$ to be chosen latter. 
Since $u\in Z^\gamma$ then a.e. $F(u(\tau))\in X_{\tau}^\gamma$ and we can apply again Proposition \ref{linearasc}. There exists $u_+(\tau)\in L^2$ such that
$$\|S(t,\tau) iF(u(\tau))-e^{i(t-\tau)\partial_x^2}iu_+(\tau)\|_{L^2}\leq C(a,\tilde\delta)\,\frac {(1+\log\tau)\tau^{\frac 12-(\gamma+\tilde\delta)}}{t^{\frac14-(\gamma+\tilde\delta)}}\,\|F(u(\tau))\|_{X_\tau^\gamma}.$$
In view of \eqref{Duhamelu+} the expression of $u_+(\tau)$ is
\begin{equation}\label{Duhamelutau+}
u_+(\tau)=F(u(\tau))+a^2\int_{\tau}^\infty e^{-is\partial_x^2}\frac{\overline{S(s,\tau)iF(u(\tau))}}{s^{1+ 2ia^2}}ds.
\end{equation}

We define
\begin{equation}\label{Duhamelf+}
f_+=u_++i\int_1^\infty e^{-i(\tau-1)\partial_x^2}u_+(\tau) \frac{d\tau}{\tau}
\end{equation}
and we have
$$u(t)-e^{i(t-1)\partial_x^2}f_+=S(t,1)\,u(1)-e^{i(t-1)\partial_x^2}u_++\int_1^t S(t,\tau) iF(u(\tau))\frac{d\tau}{\tau}-i\int_1^\infty e^{i(t-\tau)\partial_x^2}u_+(\tau)\frac{d\tau}{\tau}$$
$$=S(t,1)\,u(1)-e^{i(t-1)\partial_x^2}u_++\int_1^t \left(S(t,\tau) iF(u(\tau))-e^{i(t-\tau)\partial_x^2}iu_+(\tau)\right)\frac{d\tau}{\tau}-i\int_t^\infty e^{i(t-\tau)\partial_x^2}u_+(\tau)\frac{d\tau}{\tau}.$$
The first term has the right decay in $L^2$, and the second is upper-bounded by
$$\left\|\int_1^t \left(S(t,\tau) iF(u(\tau))-e^{i(t-\tau)\partial_x^2}iu_+(\tau)\right)\frac{d\tau}{\tau}\right\|_{L^2}\leq C(a,\tilde\delta)\int_1^t \frac {(1+\log\tau)\tau^{\frac 12-(\gamma+\tilde\delta)}}{t^{\frac14-(\gamma+\tilde\delta)}}\,\|F(u(\tau))\|_{X_\tau^\gamma}\frac{d\tau}{\tau},$$
so we can use Lemma \ref{Duh23} with $\alpha=\frac 12-(\gamma+\tilde\delta)$,
$$\left\|\int_1^t \left(S(t,\tau) iF(u(\tau))-e^{i(t-\tau)\partial_x^2}iu_+(\tau)\right)\frac{d\tau}{\tau}\right\|_{L^2}\leq C(a,\tilde\delta)\frac{1+\log t}{t^{\frac14-(\gamma+\tilde\delta)}}\,(\|u\|_{Z^\gamma}^2+\|u\|_{Z^\gamma}^3).$$
For the last term we use \eqref{u+L2}
$$\left\|\int_t^\infty e^{i(t-\tau)\partial_x^2}u_+(\tau)\frac{d\tau}{\tau}\right\|_{L^2}\leq \int_t^\infty \left\|u_+(\tau)\right\|_{L^2}\frac{d\tau}{\tau}\leq C(a)\int_t^\infty \tau^\frac 14\left\|F(u(\tau))\right\|_{X_\tau^\gamma}\frac{d\tau}{\tau},$$
and again Lemma \ref{Duh23} with $\alpha=\frac 14$
$$\left\|\int_t^\infty e^{i(t-\tau)\partial_x^2}u_+(\tau)\frac{d\tau}{\tau}\right\|_{L^2}\leq C(a)\frac{\|u\|_{Z^\gamma}^2+\|u\|_{Z^\gamma}^3}{t^{\frac14-\gamma}}.$$

In conclusion we have
$$\|u(t)-e^{i(t-1)\partial_x^2}f_+\|_{L^2}\leq C(a,\tilde\delta)\frac{1+\log t}{t^{\frac14-(\gamma+\tilde\delta)}}\,(\|u(1)\|_{X_1^\gamma}+\|u\|_{Z^\gamma}^2+\|u\|_{Z^\gamma}^3)$$
$$\leq C(a,\tilde\delta)\frac{1+\log t}{t^{\frac14-(\gamma+\tilde\delta)}}\,(\|u(1)\|_{X_1^\gamma}+\|u(1)\|_{X_1^\gamma}^2+\|u(1)\|_{X_1^\gamma}^3)\leq C(a,\tilde\delta)\frac{1+\log t}{t^{\frac14-(\gamma+\tilde\delta)}}\,\|u(1)\|_{X_1^\gamma},$$
and the Proposition follows by choosing $0<\tilde\delta<\delta<1/4-\gamma$.

\end{proof}

Similarly we get also the statement for Sobolev spaces.
\begin{corollary}\label{nonlinearascHs}
Let $s\in\mathbb N$. Let $u(1)$ be a function in $X_{1}^\gamma$ such that $\partial_x^k u(1)\in X_{1}^\gamma$ for all $0\leq k\leq s$, small with respect to $a$. Then the unique global solution $u\in Y_{1}^\gamma$ of equation \eqref{lin} with $u(1)$ initial data at time $t=1$, with $\partial_x^k u\in Y_{1}^\gamma$ for all $0\leq k\leq s$, scatters in $H^s$. 
More precisely there exists $f_+\in H^s$ such that
\begin{equation}\label{ratenonlinHs}
\|u(t)-e^{i(t-1)\partial_x^2}f_+\|_{H^s}\leq  \frac{C(a,\delta)}{t^{\frac14-(\gamma+\delta)}}\,\Sigma_{0\leq k\leq s}\,\|\partial_x^ku(1)\|_{X_1^\gamma}\underset{t\tend\infty}{\longrightarrow} 0,
\end{equation}
for any $0<\delta<1/4-\gamma$.
\end{corollary}

\subsection{Regularity of the asymptotic state}\label{ss-nonlin-aposteriori}
As an extra-information on $f_+$, we have the following result, in the spirit of \eqref{fourierlin}. This completes the proof of Theorem \ref{theorem1}.

\begin{prop}\label{estf+}
If $\|u(1)\|_{X_1^\gamma}$ is small enough with respect to $a$, the function $f_+$ satisfies for all $\xi^2\leq 1$ and $0<\delta<1/4-\gamma$
$$|\xi|^{2(\gamma+\delta)}|\hat{f_+}(\xi)|\leq C(a,\delta)\|u(1)\|_{X_1^\gamma}.$$
\end{prop}

\begin{proof}
By definition \eqref{Duhamelf+} of $f_+$, we have
$$|\xi|^{2(\gamma+\delta)}\hat{f_+}(\xi)=|\xi|^{2(\gamma+\delta)}\left(\hat{u_+}(\xi)+i\int_1^{1/\xi^2} e^{i(\tau-1)\xi^2}\hat{u_+}(\tau,\xi) \frac{d\tau}{\tau}+i\int_{1/\xi^2}^\infty e^{i(\tau-1)\xi^2}\hat{u_+}(\tau,\xi) \frac{d\tau}{\tau}\right),$$
so on $\xi^2\leq 1$ the estimate \eqref{fourierlin} insures us that the first term is upper-bounded by $C(a,\delta)\|u(1)\|_{X_1^\gamma}$.

By Lemma \ref{u+}  and \eqref{Duhamelutau+} we can treat the first integral
$$\int_1^{1/\xi^2}|\xi|^{2(\gamma+\delta)}|\hat{u_+}(\tau,\xi)| \frac{d\tau}{\tau}\leq\int_1^{1/\xi^2} \frac{C(a,\delta)}{\tau^{\delta}}|\xi|^{2\gamma}(1+|\log|\xi||)\left(|\widehat{F(u(\tau))}(\xi)|+|\widehat{F(u(\tau))}(-\xi)|\right)\frac{d\tau}{\tau} $$
$$\leq C(a,\delta)\,\int_1^{1/\xi^2} \left(\|u(\tau)\|_{L^2}^2+\|u(\tau)\|_{L^3}^3\right) \frac{1+\log\tau}{\tau^{1+\gamma+\delta}}d\tau.$$
As usual, we use H\"older's inequality for the second term, and get
$$\int_1^{1/\xi^2}|\xi|^{2(\gamma+\delta)}|\hat{u_+}(\tau,\xi)| \frac{d\tau}{\tau}\leq C(a,\delta)\,\|u\|_{L^\infty((1,\infty), L^2)}^2\int_1^{1/\xi^2}\frac{1+\log\tau}{\tau^{1+\gamma+\delta}}d\tau$$
$$+C(a,\delta)\,\|u\|_{L^{12}((1,\infty), L^3)}^3 \left\|\frac{1+\log\tau}{\tau^{1+\gamma+\delta}}\right\|_{L^\frac 43(1,{1/\xi^2})}\leq C(a,\delta)\,(\|u\|_{Z^\gamma}^2+\|u\|_{Z^\gamma}^3)\leq C(a,\delta)\|u(1)\|_{X_1^\gamma}.$$

Remains to estimate the last integral which, in view of \eqref{Duhamelutau+}, is
\begin{equation}\label{remaining}
|\xi|^{2(\gamma+\delta)}\int_{1/\xi^2}^\infty e^{i(\tau-1)\xi^2}\hat{u_+}(\tau,\xi) \frac{d\tau}{\tau}
\end{equation}
$$=|\xi|^{2(\gamma+\delta)}e^{-i\xi^2}\int_{1/\xi^2}^\infty e^{i\tau\xi^2}\left(\widehat{F(u(\tau))}(\xi)+a^2\int_\tau^\infty e^{is\xi^2}\frac{\widehat{\overline{S(s,\tau)iF(u(\tau))}}(\xi)}{s^{1+ 2ia^2}}ds\right) \frac{d\tau}{\tau}$$$$=I_1(\xi)+I_2(\xi)+I_3(\xi),$$
where we denote by $I_1$ the cubic contributions of $\widehat{F(u(\tau))}$, by $I_2$ the quadratic ones, and by $I_3$ the double integral. By Lemma \ref{IBP} applied for some $0<\tilde\delta$, since $\xi^2\leq 1$,
$$|I_3(\xi)|\leq |\xi|^{2(\gamma+\delta)}\int_{1/\xi^2}^\infty C(a)\frac{|\widehat{F(u(\tau))}(\xi)|+|\widehat{F(u(\tau))}(-\xi)|}{\xi^2\tau} \frac{d\tau}{\tau}$$
$$\leq C(a)\frac{|\xi|^{2(\gamma+\delta)}}{\xi^{2}}\int_{1/\xi^2}^\infty \left(\|u(\tau)\|_{L^2}^2+\|u(\tau)\|_{L^3}^3\right)  \frac{d\tau}{\tau^{2}}$$
$$\leq C(a)\frac{|\xi|^{2(\gamma+\delta)}}{\xi^{2}}\left(\|u\|_{L^\infty((1,\infty), L^2)}^2\int_{1/\xi^2}^\infty \frac{d\tau}{\tau^{2}}+\|u\|_{L^{12}((1,\infty), L^3)}^3\left\|\frac{1}{\tau^{2}}\right\|_{L^\frac 43({1/\xi^2},\infty)}\right)\leq C(a)\|u(1)\|_{X_1^\gamma}$$
We have on $\xi^2\leq 1$
$$|I_1(\xi)|\leq\int_{1}^\infty \|u(\tau)\|_{L^3}^3\frac{d\tau}{\tau}\leq C\,\|u\|^3_{L^{12}((1,\infty),L^3)}\leq C(a)\|u\|^3_{Z^\gamma}\leq C(a)\|u(1)\|_{X_1^\gamma}.$$
For the quadratic terms $I_2$ we first notice that quadratic powers of $u(\tau)$ can be replaced by the quadratic powers of $e^{i(\tau-1)\partial_x^2}f_+$, because, in view of Proposition \ref{nonlinearasc},
$$a\left|\int_{t_1}^{t_2} e^{i\tau\xi^2}\left(\mathcal{F}u^2(\tau,\xi)-\mathcal{F}(e^{i(\tau-1)\partial_x^2}f_+)^2(\xi)\right)\frac{d\tau}{\tau^{1-ia^2}}\right|$$
$$\leq a\int_{t_1}^{t_2} \|u^2(\tau)-(e^{i(\tau-1)\partial_x^2}f_+)^2\|_{L^1}\frac{d\tau}{\tau}\leq \frac{C(a,\delta)}{t_1^{\frac 14-(\gamma+\delta)}}\|u(1)\|_{X_1^\gamma}.$$

Therefore we have obtained for $\xi^2\leq 1$
$$|\xi|^{2(\gamma+\delta)}|\hat{f_+}(\xi)|\leq C(a,\delta)\|u(1)\|_{X_1^\gamma}$$$$+a|\xi|^{2(\gamma+\delta)}\left|\int_{1/\xi^2}^\infty e^{i\tau\xi^2}\,\frac{\mathcal{F}(e^{i(\tau-1)\partial_x^2}f_+)^2(\xi)+2\mathcal{F}|e^{i(\tau-1)\partial_x^2}f_+|^2(\xi)}{\tau^{1-ia^2}}d\tau\right|.$$
By writing explicitly the Fourier transforms, we get
\begin{equation}\label{estasstate}
|\xi|^{2(\gamma+\delta)}|\hat{f_+}(\xi)|\leq C(a,\delta)\|u(1)\|_{X_1^\gamma}+\Sigma_{j\in\{1,2\}}a|\xi|^{2(\gamma+\delta)}\int |\hat{f}_+(\eta)||\hat{f}_+(\xi-\eta)|\,\left|\int_{1/\xi^2}^\infty e^{i\tau h_j(\xi,\eta)}\frac{d\tau}{\tau^{1-ia^2}}\right|\,d\eta,
\end{equation}
where
$$h_1(\xi,\eta)=2\eta(\xi-\eta) \quad,\quad h_2(\xi,\eta)=2\xi(\xi-\eta).$$
By integrating by parts, for $\eta\neq\xi$, $\eta\neq 0$,
$$\int_{1/\xi^2}^\infty e^{i\tau\,h_j(\xi,\eta)}\,\frac{d\tau}{\tau^{1-ia^2}}=\left.\frac{e^{i\tau\,h_j(\xi,\eta)}}{ih_j(\xi,\eta)\tau^{1-ia^2}}\right|_{1/\xi^2}^\infty+(1-ia^2)\int_{1/\xi^2}^\infty \frac{e^{i\tau\,h_j(\xi,\eta)}}{ih_j(\xi,\eta)}\,\frac{d\tau}{\tau^{2-ia^2}}.$$
On one hand if $|h_j(\xi,\eta)|\geq c\,\xi^2$ for some positive constant $c$, we get the uniform estimate
$$\left|\int_{1/\xi^2}^\infty e^{i\tau\,h_j(\xi,\eta)}\,\frac{d\tau}{\tau^{1-ia^2}}\right|\leq C(a).$$ 
On the other hand, in the region $|h_j(\xi,\eta)|\leq c\,\xi^2$, the integral from $\frac{1}{\xi^2 |h_j(\xi,\eta)|}$ to infinity can be treated the same way. Finally, since
$$\left|\int_{1/\xi^2}^{\frac{1}{\xi^2\,|h_j(\xi,\eta)|}} e^{i\tau\,h_j(\xi,\eta)}\,\frac{d\tau}{\tau^{1-ia^2}}\right|\leq |\log | h_j(\xi,\eta)||,$$
we get
$$\left|\int_{1/\xi^2}^\infty e^{i\tau\,h_j(\xi,\eta)}\,\frac{d\tau}{\tau^{1-ia^2}}\right|\leq C(a)+ |\log | h_j(\xi,\eta)||\,\mathbb I_{|h_j(\xi,\eta)|\leq c|\xi|^2}.$$ 

Summarizing, we have obtained
\begin{equation}\label{estasstatebis}
|\xi|^{2(\gamma+\delta)}|\hat{f_+}(\xi)|\leq C(a,\delta)\|u(1)\|_{X_1^\gamma}+(C(a)+|\log|\xi||)|\xi|^{2(\gamma+\delta)}\int |\hat{f}_+(\eta)||\hat{f}_+(\xi-\eta)|\,d\eta
\end{equation}
$$+2a|\xi|^{2(\gamma+\delta)}\int_{|\eta|<C} |\hat{f}_+(\eta)||\hat{f}_+(\xi-\eta)|\,|\log |\eta||\,d\eta.$$
We have also used here the fact that since $|\xi|<1$ then both $|\eta|$ and $|\eta-\xi|$ are bounded in the regions $|h_j(\xi,\eta)|\leq c|\xi|^2$.

The function $f_+$ is in $L^2$ with norm bounded by $\|u(1)\|_{X_1^\gamma}$, so Cauchy-Schwarz's inequa-\\lity yields
\begin{equation}\label{estasstateter}
|\xi|^{2(\gamma+\delta)}|\hat{f_+}(\xi)|\leq C(a,\delta)\|u(1)\|_{X_1^\gamma}
\end{equation}
$$+C(a)|\xi|^{2(\gamma+\delta)}\int_{|\eta|<C}\ |\hat{f}_+(\eta)||\hat{f}_+(\xi-\eta)|\, |\log |\eta||\,d\eta.$$
On the region $\{|\eta|<C\}\cap \{\frac{|\xi|}{2}<|\eta|\}$ we can upper-bound $|\log |\eta||\leq |\log\frac{|\xi|}{2}|$ and treat this term as before, to finally get
\begin{equation}\label{estasstatefinal}
|\xi|^{2(\gamma+\delta)}|\hat{f_+}(\xi)|\leq C(a,\delta)\|u(1)\|_{X_1^\gamma}+C(a)|\xi|^{2(\gamma+\delta)}\int_{|\eta|\leq\frac{|\xi|}{2}}\ |\hat{f}_+(\eta)||\hat{f}_+(\xi-\eta)|\, |\log |\eta||\,d\eta.\end{equation}
By Cauchy-Schwarz's inequality we obtain
$$|\hat{f_+}(\xi)|^2\leq C(a,\delta)\frac{1}{|\xi|^{4(\gamma+\delta)}}\|u(1)\|^2_{X_1^\gamma}+C(a)\|u(1)\|_{X_1^\gamma}\int_{|\eta|\leq\frac{|\xi|}{2}}\ |\hat{f}_+(\xi-\eta)|^2\, \log^2 |\eta|\,d\eta.$$
For $0<r<|\xi|$ we get then
$$\left(|\hat{f_+}|^2\star \frac{1}{2r}\mathbb I_{[-r,r]}\right)(\xi)\leq \|u(1)\|_{X_1^\gamma}\left(\frac{C(a,\delta)}{2r}\int_{|\xi'|\leq r}\frac{d\xi'}{|\xi-\xi'|^{4(\gamma+\delta)}}+C(a)\,I_r(\xi)\right),$$
where
$$I_r(\xi)=\frac{1}{2r}\int_{|\xi'|\leq r} \int_{|\eta|\leq\frac{|\xi-\xi'|}{2}}\ |\hat{f}_+(\xi-\xi'-\eta)|^2\, \log^2 |\eta|\,d\eta\,d\xi'.$$
Since $\delta<\frac 14-\gamma$, we have
$$\frac{1}{2r}\int_{|\xi'|\leq r}\frac{d\xi'}{|\xi-\xi'|^{4(\gamma+\delta)}}\leq C\,\frac{|\xi+r|^{1-4(\gamma+\delta)}-|\xi-r|^{1-4(\gamma+\delta)}}{2r}.$$
For  $\frac{|\xi|}{2}<r<|\xi|$ we get immediately the upper-bound $\frac{C}{|\xi|^{4\delta}}$, while for $0<r<\frac{|\xi|}{2}$ we get the same upper-bound by noticing that $\frac{|\xi|}{2}<|\xi-\xi'|<\frac {3|\xi|}{2}$. As a consequence, for $0<r<|\xi|$ and $\xi^2\leq 1$, 
\begin{equation}\label{conv}
\left(|\hat{f_+}|^2\star \frac{1}{2r}\mathbb I_{[-r,r]}\right)(\xi)\leq \|u(1)\|_{X_1^\gamma}\left(\frac{C(a,\delta)}{|\xi|^{4(\gamma+\delta)}}+C(a)\,I_r(\xi)\right).
\end{equation}

We define for $\xi\neq 0$ and functions $h\in L^1$
$$Mh(\xi)=\sup_{0<r<|\xi|} \left( h\star \frac{1}{2r}\mathbb I_{[-r,r]}\right)(\xi)=\sup_{0<r<|\xi|} \frac{1}{2r}\int_{|\xi'|\leq r} h(\xi-\xi') \,d\xi'.$$
We get that $Mh(\xi)$ is well defined almost everywhere in $\xi$: for $r$ large we use $h\in L^1$, and for $r\tend 0$ we get $h(\xi)<\infty$ a.e. in $\xi$. As a property of this operator we have, for $h\geq 0$ and for $\phi$ even and decreasing,
\begin{equation}\label{propertyM}
\int_{|\eta|\leq |\xi|} h(\xi-\eta)\phi(\eta)\,d\eta=\Sigma_{j=0}^{+\infty}\int_{\frac{|\xi|}{2^{j+1}}\leq |\eta|\leq \frac{|\xi|}{2^j}}h(\xi-\eta)\phi(\eta)\,d\eta
\end{equation}
$$\leq \Sigma_{j=0}^{+\infty}\,\frac{|\xi|}{2^{j-1}}\phi\left(\frac{|\xi|}{2^{j+1}}\right)\,\frac{2^{j-1}}{|\xi|}\int_{\frac{|\xi|}{2^{j+1}}\leq |\eta|\leq \frac{|\xi|}{2^j}}h(\xi-\eta)\,d\eta\leq 2Mh(\xi)\,\int_{|\eta|\leq |\xi|} \phi(\eta)\,d\eta.$$

We have the following lemma.
\begin{lemma}
For $0<r<|\xi|\leq 1$ the following inequality holds
$$I_r(\xi)\leq C(a)(1+\log^2|\xi|)\,\|u(1)\|_{X_1^\gamma}^2+C|\xi|(1+\log^2|\xi|)\,M|\hat{f_+}|^2(\xi).$$
\end{lemma}
\begin{proof}
First, for $\frac{|\xi|}{4}<r<|\xi|$, we do the change of variable $\eta=\eta'-\xi'$, so
$$I_r(\xi)=\frac{1}{2r}\int_{|\xi'|\leq r} \int_{|\eta'-\xi'|\leq\frac{|\xi-\xi'|}{2}}\ |\hat{f}_+(\xi-\eta')|^2\, \log^2 |\eta'-\xi'|\,d\eta'\,d\xi'.$$
In particular, $|\eta'|\leq\frac{|\xi|}{2}+\frac 32 r\leq 2|\xi|$, and
$$I_r(\xi)\leq \int_{|\eta'|\leq 2|\xi|}\ |\hat{f}_+(\xi-\eta')|^2\frac{1}{2r}\int_{|\xi'|\leq r}\, \log^2 |\eta'-\xi'|\,d\xi'\,d\eta'$$
$$\leq C\, \int_{|\eta'|\leq 2|\xi|}\ |\hat{f}_+(\xi-\eta')|^2\,\frac{(|\eta'|+r)\log^2(|\eta|'+r)}{2r} d\eta'\leq C \,\frac{|\xi|(1+\log^2|\xi|)}{r}\,\|f_+\|_{L^2}^2,$$
so for $\frac{|\xi|}{4}<r<|\xi|$ we have the upper-bound $C(1+\log^2|\xi|)\|f_+\|_{L^2}^2$.

For $0<r<\frac{|\xi|}{4}$ we perform the same change of variable, and get
$|\eta'|\leq\frac{|\xi|}{2}+\frac 32 r\leq |\xi|$, so
$$I_r(\xi)\leq \int_{|\eta'|\leq |\xi|}\ |\hat{f}_+(\xi-\eta')|^2\frac{1}{2r}\int_{|\xi'|\leq r}\, \log^2 |\eta'-\xi'|\,d\xi'\,d\eta'.$$
In the region $|\eta'|\geq 2r$ we have $|\xi'|\leq\frac{|\eta'|}{2}$, so $|\eta'-\xi'|\geq\frac{\eta'}{2}$, and by using \eqref{propertyM}, we get the desired upper-bound
$$\int_{|\eta'|\leq |\xi|}\ |\hat{f}_+(\xi-\eta')|^2 \log^2 \frac{|\eta'|}{2}\, d \eta'\leq 2M|\hat{f_+}|^2(\xi)\int_{|\eta'|\leq |\xi|} \log^2 \frac{|\eta'|}{2}\, d \eta'\leq C|\xi|(1+\log^2|\xi|)\,M|\hat{f_+}|^2(\xi).$$
In the remaining region $|\eta'|\leq 2r$ we decompose the integral in $\eta'$ into three parts: 
$$\int_{|\eta'|\leq \min\{ |\xi|,2r\}}\ |\hat{f}_+(\xi-\eta')|^2\frac{1}{2r}\left(\int_{|\xi'|\leq \frac{|\eta'|}{2}} +\int_{\frac{|\eta'|}{2}\leq |\xi'|\leq\frac 32|\eta'|} + \int_{\frac 32|\eta'|\leq|\xi'|\leq r}\right) d\eta'=I_r^1(\xi)+I_r^2(\xi)+I_r^3(\xi).$$
In the first one, $|\eta'-\xi'|\geq \frac{|\eta'|}{2}$, so 
$$I_r^1(\xi)\leq C\int_{|\eta'|\leq \min\{ |\xi|,2r\}}\ |\hat{f}_+(\xi-\eta')|^2  \log^2 \frac{|\eta'|}{2}\, d \eta',$$
so as before we recover the upper-bound $C|\xi|(1+\log^2|\xi|)\,M|\hat{f_+}|^2(\xi)$. 
In the second region we integrate in $\xi'$, and since $\xi'$ is of the size of $\eta'$, we end up as before
$$I_r^2(\xi)\leq C\int_{|\eta'|\leq \min\{ |\xi|,2r\}}\ |\hat{f}_+(\xi-\eta')|^2 \,\frac{|\eta'|\log^2|\eta'|}{2r}\,d\eta'\leq C\int_{|\eta'|\leq \min\{ |\xi|,2r\}}\ |\hat{f}_+(\xi-\eta')|^2  |\log^2 |\eta'|\, d \eta'.$$
In the last region $|\eta'-\xi'|\geq \frac{|\eta'|}{2}$, so we get again
$$I_r^3(\xi)\leq C\int_{|\eta'|\leq \min\{ |\xi|,2r\}}\ |\hat{f}_+(\xi-\eta')|^2 \,\log^2\frac{|\eta'|}{2}\,d\eta'.$$

In conclusion for $\frac{|\xi|}{4}<r<|\xi|$ we get the upper-bound $C(1+\log^2|\xi|)\|f_+\|_{L^2}^2$ and for $0<r<\frac{|\xi|}{4}$ we get the upper-bound $C|\xi|(1+\log^2|\xi|)\,M|\hat{f_+}|^2(\xi)$, so the Lemma follows.
\end{proof}

By using this Lemma, estimate \eqref{conv} gives us for $0<r<|\xi|$, 
$$\left(|\hat{f_+}|^2\star \frac{1}{2r}\mathbb I_{[-r,r]}\right)(\xi)\leq \|u(1)\|_{X_1^\gamma}\left(\frac{C(a,\delta)}{|\xi|^{4(\gamma+\delta)}}\right.$$
$$\left.+C(a)\,(1+\log^2|\xi|)\|u(1)\|_{X_1^\gamma}^2+C(a)\,|\xi|(1+\log^2|\xi|)\,M|\hat{f_+}|^2(\xi)\right).$$
The constant is independent of $r$, so by taking the supremum in $0<r<|\xi|$ we obtain for $\xi^2\leq 1$,
\begin{equation}\label{convbis}
M|\hat{f_+}|^2(\xi)\leq \|u(1)\|_{X_1^\gamma}\left(\frac{C(a,\delta)}{|\xi|^{4(\gamma+\delta)}}+C(a)\,(1+\log^2|\xi|)\|u(1)\|_{X_1^\gamma}^2+C(a)\,|\xi|(1+\log^2|\xi|)\,M|\hat{f_+}|^2(\xi)
\right).
\end{equation}
Since $M|\hat{f_+}|^2(\xi)<\infty$ almost everywhere in $\xi$, for $\|u(1)\|_{X_1^\gamma}C(a)|\xi|(1+\log^2|\xi|)<\frac 12$, so for $C(a)\|u(1)\|_{X_1^\gamma}<\frac 12$, we get the estimate
$$M|\hat{f_+}|^2(\xi)\leq \frac{C(a,\delta)}{|\xi|^{4(\gamma+\delta)}}\|u(1)\|_{X_1^\gamma}+C(a)\,(1+\log^2|\xi|)\|u(1)\|_{X_1^\gamma}^3\leq \frac{C(a,\delta)}{|\xi|^{4(\gamma+\delta)}}\|u(1)\|_{X_1^\gamma}.$$
Then, 
$$|\hat{f_+}|^2(\xi)=\lim_{r\tend 0} \left( |\hat{f_+}|^2\star \frac{1}{2r}\mathbb I_{[-r,r]}\right)(\xi)\ \leq M|\hat{f_+}|^2(\xi) \leq \frac{C(a,\delta)}{|\xi|^{4(\gamma+\delta)}}\|u(1)\|_{X_1^\gamma},$$
and the Proposition follows.
\end{proof}

\appendix
\section{Wave operators}
In this section we prove the existence of wave operators for the nonlinear equation \eqref{nonlin}. The difference with respect to the wave operators constructed in \cite{BV} is that here we shall weaken the conditions on the final data by working in spaces that fits with the conditions of Theorem \ref{theorem1}.

We first study the existence of the wave operators for the linearized equation \eqref{lin}.

\begin{prop}\label{waveoplin}
Let $0\leq\gamma<\frac 14$, $0<\nu$, and let $u_+\in X_1^{\gamma-\nu}$. Then the  equation (\ref{lin}) has a unique solution $u\in Z^\gamma$ satisfying as $t$ goes to infinity,  
\begin{equation*}\label{ratewavelinbis}
\| u(t)-e^{it\partial_x^2}u_+\|_{L^2}\leq \,\frac{C(a,\nu,\delta)}{t^{\frac14-(\gamma+\delta)}}\,\|u_+\|_{X_1^{\gamma-\nu}},
\end{equation*}
for any $0<\delta<\frac 14-\gamma$. In particular we have $u(1)\in X_1^\gamma$, with norm bounded by $\|u_+\|_{X^{\gamma-\mu}_1}$.
\end{prop}

\begin{proof}
We are going to use similar arguments as those in Lemma \ref{controls2}. We define as in \eqref{defu+}
$$2 \mathring Z^+_\xi=e^{-ia^2\log\xi^2}\hat{u_+}(\xi),\quad  \mathring Y^+_\xi=\overline{ \mathring Z^+_{-\xi}}.$$
We define for $0\leq \tilde t\leq\frac{1}{4a^2}$ the solutions $(\mathring y_\xi,\mathring z_\xi)(\tilde t)$ of 
$$\left(\begin{array}{c}\mathring y_\xi(\tilde t)\\ \mathring z_\xi(\tilde t)\end{array}\right)=\left(\begin{array}{c}\mathring Y^+_\xi\\ \mathring Z^+_\xi\end{array}\right)+\int_0^{\tilde t} M\left(\frac{1}{\tau}\right)\,\left(\begin{array}{c}\mathring y_\xi(\tau)\\ \mathring z_\xi(\tau)\end{array}\right)\,\left(-\frac{1}{\tau^2}\right)d\tau,$$
Then
$$\sup_{0\leq \tilde t\leq \frac{1}{4a^2}}\left(|\mathring y_\xi(\tilde t)|+|\mathring z_\xi(\tilde t)|\right)\leq \left(|\mathring Y^+_\xi|+|\mathring Z^+_\xi|\right)+\int_0^\frac{1}{4a^2} \frac{a^2}{\alpha^2(1/\tau)}d\tau \sup_{0\leq \tilde t\leq \frac{1}{4a^2}}\left(|\mathring y_\xi(\tilde t)|+|\mathring z_\xi(\tilde t)|\right),$$
and as $\alpha(1/\tau)=\sqrt{1-2a^2\tau}$, we get
$$\sup_{0\leq \tilde t\leq \frac{1}{4a^2}}\left(|\mathring y_\xi(\tilde t)|+|\mathring z_\xi(\tilde t)|\right)\leq 2\left(|\mathring Y^+_\xi|+|\mathring Z^+_\xi|\right).$$
Now, for $4a^2\leq t< \infty$, the functions $(\mathring Y_\xi(t),\mathring Z_\xi(t))=(\mathring y_\xi(1/t),\mathring z_\xi(1/t))$ solve  \eqref{asmodes} and 
$$|\mathring Y_\xi(t)|^2+|\mathring Z_\xi(t)|^2\leq C\left(|\mathring Y^+_\xi|^2+|\mathring Z^+_\xi|^2\right)= C(|\hat{u_+}(\xi)|^2+|\hat{u_+}(-\xi)|^2).$$
In particular, 
$$\partial_t\left(\begin{array}{c}\mathring Y_\xi\\ \mathring Z_\xi\end{array}\right)=M(t)\left(\begin{array}{c}\mathring Y_\xi\\ \mathring Z_\xi\end{array}\right)=\frac{a^2}{2t^2\alpha^2(t)}
\left(\begin{array}{cc} -1 & e^{-2i\Phi(t)}\\e^{2i\Phi(t)} & -1\end{array}\right)\left(\begin{array}{c}\mathring Y_\xi\\ \mathring Z_\xi\end{array}\right).$$
Since $\mathring Y^+_\xi=\overline{ \mathring Z^+_{-\xi}}$ and
$$\partial_t\left(\begin{array}{c}\mathring Y_\xi-\overline{ \mathring Z_{-\xi}}\\ \mathring Z_\xi-\overline{ \mathring Y_{-\xi}}\end{array}\right)=M(t)\left(\begin{array}{c}\mathring Y_\xi-\overline{ \mathring Z_{-\xi}}\\ \mathring Z_\xi-\overline{ \mathring Y_{-\xi}}\end{array}\right),$$
we obtain that $\mathring Y_\xi(t)=\overline{ \mathring Z_{-\xi}(t)}$. Therefore, with the notations of Lemma \ref{controls2}, we can define for $4a^2\leq t< \infty$, 
$$\left(\begin{array}{c}Y_\xi(t)\\ Z_\xi(t)\end{array}\right)
=P(t)\left(\begin{array}{cc} e^{i\Phi(t)} &0\\0 & e^{-i\Phi(t)}\end{array}\right)\left(\begin{array}{c}\mathring Y_\xi(t)\\ \mathring Z_\xi(t)\end{array}\right)=\left(\begin{array}{c}e^{i\Phi(t)} \mathring Y_\xi(t)+e^{-i\Phi(t)} \mathring Z_\xi(t)\\ i\alpha(t)e^{i\Phi(t)} \mathring Y_\xi(t)-i\alpha(t)e^{-i\Phi(t)}\mathring Z_\xi(t)\end{array}\right),$$ 
solution of \eqref{system}:
$$\partial_t\left(\begin{array}{c}Y_\xi\\ Z_\xi\end{array}\right)
=\left(\begin{array}{cc}0 &1\\-\left(1-\frac{2a^2}{t}\right) & 0\end{array}\right)\left(\begin{array}{c}Y_\xi\\ Z_\xi\end{array}\right),$$
which satisfies $Y_\xi(t)=\overline{ Y_{-\xi}(t)}$ and $Z_\xi(t)=\overline{ Z_{-\xi}(t)}$. 
Moreover, since
$$|\mathring Y_\xi(t)|^2+|\mathring Z_\xi(t)|^2=\left|\frac 12 Y_\xi(t)-\frac{i}{2\alpha(t)}Z_\xi(t)\right|^2+\left|\frac 12 Y_\xi(t)+\frac{i}{2\alpha(t)}Z_\xi(t)\right|^2= \frac{|Y_\xi(t)|^2}{2}+\frac{|Z_\xi(t)|^2}{2\alpha^2(t)},$$
and from $\frac{1}{\sqrt{2}}\leq \alpha(t)\leq 1$ it follows that
\begin{equation}\label{rateapp}
|Y_\xi(t)|^2+|Z_\xi(t)|^2\leq C(|\hat{u_+}(\xi)|^2+|\hat{u_+}(-\xi)|^2).
\end{equation}
We continue the definition of $\left(Y_\xi(t),Z_\xi(t)\right)$ for the remaining $0<t<\infty$ as solution of \eqref{system}. It follows that $u(t,x)$ defined by $u(t,x)=w(t,x)e^{- ia^2\log t}$, where
$$Y_\xi( t)=\widehat{\Re \,w}\left(\frac{t}{\xi^2},\xi\right)\,\,,\,\,Z_\xi( t)=\widehat{\Im \,w}\left(\frac{t}{\xi^2},\xi\right),$$
is a solution of \eqref{lin}. In particular, \eqref{rateapp} is satisfied for all $1\leq t <\infty$, so we get for large frequencies $\xi^2\geq \frac{1}{t}$ that
\begin{equation}\label{largefreq}
|\hat{u}(t,\xi)|^2+|\hat{u}(t,-\xi)|^2\leq C(a)(|\hat{u_+}(\xi)|^2+|\hat{u_+}(-\xi)|^2).
\end{equation}

We define next
$$y_\xi(t)=Y_\xi\left(\frac 1t\right),\quad z_\xi(t)=Z_\xi\left(\frac 1t\right),$$
solution for $1\leq t<\infty$ of 
$$y'_\xi=-\frac{1}{t^2}z_\xi,\quad z'_\xi=\left(\frac{1}{t^2}-\frac{2a^2}{t}\right)y_\xi,$$
with initial data $(y_\xi(1),z_\xi(1))=(Y_\xi(1),Z_\xi(1))$. We take $\sigma_\epsilon=y_\xi^2/\epsilon  +\epsilon z_\xi^2$ and proceeding as in Lemma \ref{controls2}, for all $1< t$,
$$\sigma_\epsilon(t)\leq e^{\frac 1\epsilon+\epsilon+2a^2\epsilon\log t}\sigma_\epsilon(1).$$
By making for $t>3/2$ the choice $\epsilon=\frac{1}{\sqrt{\log t}}$ and for $1\leq t\leq 3/2$ the choice $\epsilon=1$, we get for all $1\leq t$ the estimate 
$$|y_\xi(t)|^2+|z_\xi(t)|^2\leq C(1+\log t)e^{2+2a^2\sqrt{\log t}}\left(|y_\xi(1)|^2+|z_\xi(1)|^2\right)$$
$$\leq C(a)(1+\log t)e^{2+2a^2\sqrt{\log t}}(|\hat{u_+}(\xi)|^2+|\hat{u_+}(-\xi)|^2).$$
So we finally get the estimates for low frequencies $\xi^2\leq \frac{1}{t}$ of the solution $u$ of \eqref{lin}
\begin{equation}\label{lowfreq}
|\hat{u}(t,\xi)|^2+|\hat{u}(t,-\xi)|^2\leq C(a)(1+|\log t\xi^2|)e^{2+2a^2\sqrt{|\log t\xi^2|}}(|\hat{u_+}(\xi)|^2+|\hat{u_+}(-\xi)|^2).
\end{equation}
Therefore calling $f_+(x)=u(1,x)$ we have obtained by \eqref{largefreq}  and \eqref{lowfreq} that $f_+\in X_1^\gamma$ with norm bounded by $\|u_+\|_{X^{\gamma-\mu}_1}$. From Propositions \ref{gl}, \ref{linearasc}, and \eqref{defu+}, \eqref{u+modes} it follows that $u=S(t,1)f_+\in Z^\gamma$ and for all $0<\delta<\frac 14-\gamma$
$$\|S(t,1)f_+-e^{i(t-1)\partial_x^2}u_+\|_{L^2}\leq \,\frac{C(a,\nu,\delta)}{t^{\frac14-(\gamma+\delta)}}\,\left(\|u_+\|_{L^2}+\||\xi|^{2(\gamma-\nu)}\hat{u}_+(\xi)\|_{L^\infty(|\xi|\leq 1)}\right).$$
\end{proof}

\begin{remark}Let us notice that \eqref{largefreq} together with \eqref{lowfreq} imply that for $t_1>0$ and $\delta_1>0$
$$|\hat{u}(t_1,\xi)|\leq \left(C(a)+\frac{C(a,\delta_1)}{(t_1\xi^2)^{\delta_1}}\right)(|\hat{u_+}(\xi)|+|\hat{u_+}(-\xi)|).$$
This, combined with Lemma \ref{u+}, yields, for any times $t_1,t_2\geq 1$ and for any positive $\delta_1,\delta_2$,
\begin{equation}\label{backwardest}
|\hat{u}(t_1,\xi)|\leq \left(C(a)+\frac{C(a,\delta_1)}{(t_1\xi^2)^{\delta_1}}\right)\left(C(a)+\frac{C(a,\delta_2)}{(t_2\xi^2)^{\delta_2}}\right)(|\hat{u}(t_2,\xi)|+|\hat{u}(t_2,-\xi)|).
\end{equation}
It follows that for $1\leq t_1\leq t_2$ and $\gamma+\delta_1+\delta_2<\frac 14$,
\begin{equation}\label{backwardestL2}
\|u(t_1)\|_{L^2}\leq \|u(t_2)\|_{L^2}\left(C(a)+C(a,\delta_1)\frac{t_2^{\delta_1}}{t_1^{\delta_1}}\right)
\end{equation}
$$+C(a,\delta_2)\frac{\||\xi|^{2\gamma}\hat{u}(t_2,\xi)\|_{L^\infty(\xi^2\leq\frac{1}{t_2})}}{t_1^{\delta_1}t_2^{\delta_2}}\left\|\frac{1}{\xi^{2\gamma+2(\delta_1+\delta_2)}}\right\|_{L^2(\xi^2\leq\frac{1}{t_2})}\leq C(a,\delta_1,\delta_2)\frac{t_2^{\frac 14+\delta_1}}{t_1^{\delta_1}}\|u(t_2)\|_{X^\gamma_{t_2}}.$$

\end{remark}

Now we show the existence of wave operators for the nonlinear equation \eqref{nonlin} with respect to the linear solutions of \eqref{lin}.

\begin{prop}\label{waveopnonlin}Let $0\leq\gamma<\frac 14$. For all $f_+\in X_1^\gamma$, small with respect to $a$, equation (\ref{nonlin}) has a unique solution $u \in L^\infty((1,\infty),L^2(\mathbb R))\cap L^4((1,\infty),L^\infty(\mathbb R))$ satisfying as $t$ goes to infinity 
$$\|u(t)-S(t,1)f_+\|_{L^2}+\|u(\tau)-S(\tau,1)f_+\|_{L^4((t,\infty),L^\infty)}\leq \frac{C(a,\delta)}{t^{\frac 14-(\gamma+\delta)}}\|f_+\|_{X_1^\gamma},$$
for any $0<\delta<\frac 14-\gamma$. 
\end{prop}

\begin{proof}
We shall perform a fixed point argument for the operator
$$Bu=S(t,1)f_++\int_t^{\infty}S(t,\tau)\frac{iF(u(\tau))}{\tau}d\tau $$
in the closed ball
$$A_R=\left\{u\,|\,\|u\|_A=\underset{t\in[1,\infty)}{\sup}\,t^{\frac 14-(\gamma+\delta)}\,\left(\| u(t)-S(t,1)f_+\|_{L^2}+\|u(\cdot)-S(\cdot,1)f_+\|_{L^4((t,\infty),L^\infty)}\right)\leq R\right\},$$
with $R$ to be precised later.

Let $u\in A_R$. In particular we have for all admissible couples $(p,q)$, interpolated between $(\infty,2)$ and $(4,\infty)$, 
$$\underset{t\in[1,\infty)}{\sup}\,t^{\frac 14-(\gamma+\delta)}\,\|u(\cdot)-S(\cdot,1)f_+\|_{L^p((t,\infty),L^q)}\leq CR,$$
and therefore, by the estimates \eqref{estL2} and \eqref{eststr},
\begin{equation}\label{Strboundter}
\|u\|_{L^p((t,\infty),L^q)}\leq C\|S(\cdot,1)f_+\|_{L^p((1,\infty),L^q)}+C\|u\|_A \leq C\|f_+\|_{X_1^\gamma}+C\|u\|_A.
\end{equation}

We want to estimate 
$$Bu-S(t,1)f_+=\int_t^{\infty}S(t,\tau)\frac{iF(u(\tau))}{\tau}\,d\tau=J(t).$$
We have 
$$\| J(t)\|_{L^2}+\|J\|_{L^4((t,\infty),L^\infty)}\leq  \int_t^{\infty}\|S(t,\tau)iF(u(\tau))\|_{L^2}\frac{d\tau}{\tau}+\int_t^{\infty}\|S(\cdot,\tau)iF(u(\tau))\|_{L^4((\tau,\infty),L^\infty)}\frac{d\tau}{\tau}$$
$$+\int_t^{\infty}\|S(\cdot,\tau)iF(u(\tau))\|_{L^4((t,\tau),L^\infty)}\frac{d\tau}{\tau}.$$
We upper-bound the first term by using the backwards estimates \eqref{backwardestL2} with $(t_1,t_2)=(t,\tau)$ and $\delta_1=\delta\in]0,\frac14-\gamma[$,
$$\int_t^{\infty}\|S(t,\tau)iF(u(\tau))\|_{L^2}\frac{d\tau}{\tau}\leq C(a,\delta)\int_t^{\infty}\frac{\tau^{\frac14+\delta}}{t^{\delta}}\|F(u(\tau))\|_{X_\tau^\gamma}\frac{d\tau}{\tau}.$$
For the second we use the forward estimate \eqref{eststr}, 
$$\int_t^{\infty}\|S(\cdot,\tau)iF(u(\tau))\|_{L^4((\tau,\infty),L^\infty)}\frac{d\tau}{\tau}\leq C(a)\int_t^{\infty}\tau^\frac 14(1+\log^2 \tau)\left\|F(u(\tau))\right\|_{X_\tau^{\gamma}}\,\frac{d\tau}{\tau}$$
We write the third term as
$$\int_t^{\infty}\|S(\cdot,\tau)iF(u(\tau))\|_{L^4((t,\tau),L^\infty)}\frac{d\tau}{\tau}$$
$$=\int_t^{\infty}\left\|e^{i(s-\tau)\partial_x^2}F(u(\tau))-ia^2\int_s^\tau e^{i(\tau-r)\partial_x^2}\frac{\overline{S(r,\tau)iF(u(\tau))}}{r^{1+ 2ia^2}}dr\right\|_{L^4((t,\tau),L^\infty)}\frac{d\tau}{\tau}.$$
We use the Strichartz estimates and the backwards estimates \eqref{backwardestL2},
$$\int_t^{\infty}\|S(\cdot,\tau)iF(u(\tau))\|_{L^4((t,\tau),L^\infty)}\frac{d\tau}{\tau}\leq C(a)\int_t^{\infty}\|F(u(\tau))\|_{L^2}+\left\|\frac{S(r,\tau)iF(u(\tau))}{r}\right\|_{L^1((t,\tau),L^2)}\frac{d\tau}{\tau}$$
$$\leq \int_t^{\infty}\|F(u(\tau))\|_{X_\tau^\gamma}\left(C\tau^\frac 14+C(a)\left\|\frac{\tau^{\frac14+\delta}}{r^{1+\delta}}\right\|_{L^1(t,\tau)}\right)\frac{d\tau}{\tau}\leq C(a)\int_t^{\infty}\frac{\tau^{\frac 14+\delta}}{t^\delta}\|F(u(\tau))\|_{X_\tau^\gamma}\frac{d\tau}{\tau}.$$
Summarizing, we have obtained that
$$\| J(t)\|_{L^2}+\|J\|_{L^4((t,\infty),L^\infty)}\leq C(a,\delta)\int_t^{\infty}\tau^{\frac14+\delta}\|F(u(\tau))\|_{X_\tau^\gamma}\frac{d\tau}{\tau}.$$

Now Lemma \ref{Duh23} with $(t_1,t_2)=(t,\infty)$ and $\alpha=\frac14+\delta$ gives
$$\| J(t)\|_{L^2}+\|J\|_{L^4((t,\infty),L^\infty)}\leq \frac{C(a,\delta)}{t^{\frac 14-(\gamma+\delta)}}\,\Sigma_{j\in\{1,2\}}\left(a\|u\|^2_{L^{p_j}((t,\infty),L^{q_j})}+\|u\|^3_{L^{p_j}((t,\infty),L^{q_j})}\right),$$

where $(p_1,q_1)=(\infty,2)$ and $(p_2,q_2)=(4,\infty)$.
Therefore, in view of \eqref{Strboundter}
$$\|J\|_{A}\leq Ca\,\|f_+\|_{X_1^\gamma}^2+Ca\,\|u\|_A^2+C\,\|f_+\|_{X_1^\gamma}^3+C\,\|u\|_A^3.$$
For all $f_+\in X_1^\gamma$ small with respect to $a$, there exists $R$ small with respect to $a$, such that the operator $B$ is a contraction on $A_R$, and the Proposition follows. 
\end{proof}

The last two propositions imply the following result.
\begin{theorem}\label{waveopnonlinfinal}
Let $0\leq\gamma<\frac 14$, $0<\nu$ and $u_+\in X_1^{\gamma-\nu}$ with norm small with respect to $a$. Then the  equation (\ref{nonlin}) has a unique solution $u\in L^\infty((1,\infty),L^2(\mathbb R))\cap L^4((1,\infty),L^\infty(\mathbb R))$ satisfying as $t$ goes to infinity 
\begin{equation*}\label{ratewavelinbis}
\| u(t)-e^{i(t-1)\partial_x^2}u_+\|_{L^2}\leq \,\frac{C(a,\nu,\delta)}{t^{\frac14-(\gamma+\delta)}}\,\|u_+\|_{X_1^{\gamma-\nu}},
\end{equation*}
for any $0<\delta<\frac 14-\gamma$ .
\end{theorem}

\section{Remarks on the growth of the zero-Fourier modes}\label{s:appendix-fourier}

\subsection{Growth of the zero-Fourier modes for the linear equation}\label{ss:appendix-fourier-lin} Let $u$ be the global $H^2$ solution of \eqref{lin} obtained as a consequence of Lemma \ref{controls}. 
We shall get here some extra-information on $u(t)$, via estimates done directly on $w(t)=u(t)e^{\pm ia^2\log t}$ the solution of \eqref{linearized}: 
$$i\partial_tw+w_{xx}\pm\frac{a^2}{t}(w+\overline{w})=0.$$
We shall use the fact that $w\in H^2$ to get proper integration by parts at the level of the Laplacian. 

Let us notice that since $u$ is a solution of the linear equation \eqref{lin}, if $\hat u(t_0)$ is continuous, so will be $\hat u(t)$. In this case, by integrating in space, we get the law of evolution of the zero-Fourier modes,
$$i\partial_t\int w=\mp \frac{a^2}{t}\int\Re w,$$
so
$$\partial_t\int\Re w=0\quad,\quad\partial_t\int\Im w=\pm\frac{a^2}{t}\int\Re w=\pm\frac{a^2}{t}\int\Re w(t_0).$$
Therefore
$$\int\Im  w(t)=\int\Im  w(t_0)\pm2a^2 \int\Re  w(t_0) \log\frac{t}{t_0}.$$
In conclusion, if the zero-mode $\int  w (t_0)$ is null, then it will be the same for all times,
$\int w (t)=0.$
Furthermore, if the real part of the zero-modes $\Re\int  w (t_0)$ is not null, then we have a logarithmic growth of the zero-modes $\int  w (t)$, independently of the size of $t_0$, that cannot be avoided,
\begin{equation}\label{evolutionzeromodesomega}
\int w(t)=\int w(t_0)\pm2ia^2\int\Re w(t_0) \,\log\frac{t}{t_0}.
\end{equation}
Recovering the expression of $u$, we obtain \eqref{evolutionzeromodes}.
\\

\subsection{Growth of the Fourier modes for the nonlinear equation} \label{ss:appendix-fourier-nonlin}

Let $u$ be the global $H^1$ solution of \eqref{nonlin} obtained by Corollary \ref{globalHs}. 
In particular,
$$\Sigma_{0\leq k\leq 1}\|\partial_x^k u\|_{Z}\leq C(a)\, \Sigma_{0\leq k\leq 1}\|\partial_x^k u(1)\|_{X_1}\leq C(a,u(1)).$$
\\
For the computations on Fourier modes in this subsection, the existence of $\hat u(t,0)$ has to be justified. We have the following control.

\begin{lemma}
If $ x\,u(1)\in L^2$, then 
\begin{equation*}\label{weights}
\|xu(t)\|_{L^2}\leq \,C(a,u(1))\,t^{\tilde C(a,u(1))}.
\end{equation*}
\end{lemma}

\begin{proof}
Let $\varphi$ be a positive radial cutoff function, equal to $x^2$ on $B(0,1)$, such that $(\partial_x\varphi)^2\leq C\varphi$. For $R>0$ we define 
$$\varphi_R(x)=R^2\varphi\left(\frac{x}{R}\right).$$

We multiply equation \eqref{nonlin} by $\varphi_R\overline u$ and integrate the inaginary part,
$$\partial_t\int\varphi_R\,|u(t)|^2=-\Im\int u_{xx}\,\varphi_R\,\overline u\mp \Im\, \int \frac{a^2}{t^{1\pm 2ia^2}}\,\varphi_R\,\overline{u}\,\overline{u}- \Im\int\frac{F(u)}{t}\,\varphi_R\,\overline{u}$$
$$=\Im\int u_{x}\,\partial_x\varphi_R \,\overline u\mp \Im\, \int \frac{a^2}{t^{1\pm 2ia^2}}\,\varphi_R\,\overline{u}\,\overline{u}- \Im\int\frac{F(u)}{t}\,\varphi_R\,\overline{u}$$
$$\leq \|\partial_x u\|_{L^2}\left(\int(\partial_x\varphi_R)^2\, |u(t)|^2\right)^\frac 12+ \frac{a^2}{t}\int\varphi_R\,|u(t)|^2+\frac{\|u\|_{L^\infty}+\|u\|_{L^\infty}^2}{t}\int\varphi_R\,|u(t)|^2.$$
Therefore, by using $(\partial_x\varphi)^2\leq C\varphi$ and Sobolev embeddings, 
$$\partial_t\left(\int\varphi_R\,|u(t)|^2\right)^\frac 12\leq C(a,u(1))+\frac{C(a,u(1))}{t}\left(\int\varphi_R\,|u(t)|^2\right)^\frac 12,$$
so
$$\left(\int\varphi_R\,|u(t)|^2\right)^\frac 12\leq C(a,u(1))\,t^{\tilde C(a,u(1))}.$$
The estimate is uniformly in $R$, and the Lemma follows by letting $R$ goes to infinity.
\end{proof}

In particular, the Lemma insures us that $\hat u(t)\in H^1$, so in particular $\hat u(t)$ is continuous and the existence of $\hat u(t,0)$ is justified. 
Now we shall get informations on the zero-mode of $u(t)$, via estimates on $w$ the solution of \eqref{GP}:
$$iw_t+w_{xx}=\mp\D\frac 1t\left(|a+w|^2-a^2\right)(a+w).$$
We shall use the following conservation law
\begin{equation}\label{conslaw}
\partial_t\int (| w+a|^2-a^2)=0,
\end{equation}
obtained by multiplying \eqref{GP} by $\overline{ w}+a$ and by taking the imaginary part. We integrate in space \eqref{GP} to get
$$i\partial_t\int w\pm\int\frac 1t(| w+a|^2-a^2)( w+a)=0.$$
By using \eqref{conslaw} we get the evolution of the zero-modes
$$\int w(t)-\int w(t_0)= \pm i  \int_{t_0}^t\int(| w(\tau)+a|^2-a^2)( w(\tau)+a)dx\,\frac{d\tau}{\tau}$$
$$=\pm i a\int(| w(t_0)+a|^2-a^2)dx\,\log{\frac{t}{t_0}}\pm i \int_{t_0}^t\int(| w(\tau)|^2+2a\Re w(\tau)) w(\tau)dx\,\frac{d\tau}{\tau}.$$
The Strichartz estimates imply that the part coming from the cubic power of $ w$ is bounded in time, so we can bound the second term,
$$\left|\int_{t_0}^t\int(| w(\tau)|^2+2a\Re w(\tau)) w(\tau)dx\,\frac{d\tau}{\tau}\right|\leq C(a)\|u(t_0)\|_{X_{t_0}}+2a\| w\|^2_{L^\infty((t_0,t), L^2)}\,\log{\frac{t}{t_0}}$$
$$\leq  C(a)\|u(t_0)\|_{X_{t_0}}+C(a)\|u(t_0)\|_{X_{t_0}}^2\,\log{\frac{t}{t_0}}.$$
Therefore we get a logarithmic upper-bound for $\int w(t)$, and implicitly for $\hat u(t,0)$. This growth is sharp provided that
$$C(a)\| w(t_0)\|_{X_{t_0}}^2=C(a,t_0)\left(\| w(t_0)\|_{L^2}^2+\|\hat w(t_0)\|_{L^\infty(\xi^2\leq \frac {1}{t_0})}^2\right)<a\left|\int(| w(t_0)+a|^2-a^2)dx\right|,$$
for which a sufficient condition is
$$C(a,t_0)\left(\| w(t_0)\|_{L^2}^2+\|\hat w(t_0)\|_{L^\infty(\xi^2\leq \frac {1}{t_0})}^2\right)<\left|\int \Re w(t_0)dx\right|.
$$
We get also a logarithmic growth for $\Im \int w(t)$, provided that
$\int(| w(t_0)+a|^2-a^2)dx>0.$

\end{document}